\documentclass[leqno]{siamltex704}
\usepackage{amsmath}
\usepackage{graphicx}
\usepackage[notcite,notref]{}
\usepackage{float}
\usepackage{amsfonts,amssymb}
\usepackage{hyperref}

\usepackage{enumitem}
\usepackage{color}         
\usepackage{overpic}
\usepackage{subfig}

\DeclareMathOperator*{\argmin}{arg~min}

\newcommand{\bff}{{\bf f}}
\newcommand{\bfx}{{\bf x}}
\newcommand{\bfv}{{\bf v}}
\newcommand{\sign}{\hbox{sign}}
\newcommand{\bfc}{{\bf c}}
\newcommand{\bv}[1]{\boldsymbol{#1}}
\newcommand{\R}{\mathbb{R}}

\newcommand{\cN}{{\cal N}}
\newcommand{\T}{{\cal T}}

\newcommand{\W}[2]{\mathrm{W}^{#1,#2}(\Omega)}

\def\jump#1{{[\![#1]\!]}}
\newcommand{\E}{{\cal E}}
\def\3bar{{|\hspace{-.02in}|\hspace{-.02in}|}}

\numberwithin{equation}{section}

\newtheorem{example}{\bf Example}[section]

\newtheorem{algorithm}{Algorithm}[section]

\title{A Bivariate Spline Method for Second Order Elliptic Equations
in Non-Divergence Form}

\author{Ming-Jun Lai\thanks{Department of Mathematics,
University of Georgia, Athens, GA 30602. mjlai@uga.edu. This
research is partially supported by Simons collaboration grant 280646
and the National Science Foundation under the grant \#DMS 1521537. }
\and Chunmei Wang\thanks{Department of Mathematics, Texas State
University, San Marcos, Texas 78666
(c\_w280@txstate.edu).  The research of Chunmei Wang was partially supported by National Science Foundation 
Award DMS-1522586, National Natural Science Foundation of China Award \#11526113, 
Jiangsu Key Lab for NSLSCS Grant \#201602, and by Jiangsu Provincial Foundation Award \#BK20050538.} }

\begin{document}

\maketitle

\begin{abstract}
A bivariate spline method  
is developed to numerically solve second order
elliptic partial differential equations (PDE) in non-divergence form. 
The existence, uniqueness,  stability  as well as approximation 
properties of the discretized solution will be established by using the 
well-known Ladyzhenskaya-Babuska-Brezzi (LBB) condition. 
Bivariate splines, discontinuous splines with smoothness constraints 
are used to implement the method. A plenty of computational results 
based on splines of various degrees are presented to demonstrate the 
effectiveness and efficiency of our method.
\end{abstract}



\pagestyle{myheadings}

\section{Introduction} 
We are interested in developing an efficient numerical method for
solving  second order elliptic equations in non-divergence form. To
this end,   consider the   model problem:  Find $u=u(x)$ satisfying
\begin{align}
\label{model}
 \sum_{i,j=1}^2 a_{ij}\partial_{ij}^2 u+cu&=f, \quad \mbox{in } \Omega,\\
u&=0, \quad \mbox{on } \partial\Omega, \label{model:bdrycond}
\end{align}
where $\Omega$ is an open bounded domain in $\mathbb R^2$
with a Lipschitz continuous boundary $\partial\Omega$, $\partial_{ij}^2 $
is the second order partial derivative operator with respect to $x_i$
and $x_j$ for $i,j=1, 2$.
Assume that the tensor $a(x)=\{a_{ij}(x)\}_{2\times 2}$ is
symmetric positive definite and uniformly bounded over $\Omega$, the 
coefficient $c(x)$ is non-positive and uniformly bounded over $\Omega$. 
In addition, we assume that the coefficients $a_{ij}(x)$  are 
essentially bounded so that the second order model problem 
(\ref{model}) cannot be
rewritten in a divergence form  which many existing numerical
methods can be employed to solve.

For convenience, we shall assume that the model problem (\ref{model})
has a unique strong solution $u\in H^2(\Omega)$ 
satisfying the $H^2$ regularity:
\begin{equation}
\label{regularity}
\| u\|_{H^2(\Omega)}\le C \|f\|_{L^2(\Omega)}
\end{equation}
for a positive constant $C$. Particularly, when $\Omega$ is a convex
domain with $C^2$ smooth boundary, one can show that there exists a strong solution $u\in  H^2(\Omega)$
of (\ref{model})  satisfying the $H^2$ regularity (\ref{regularity})
if  the coefficient  $c(x)\equiv 0$  and  the coefficients $a_{ij} 
(i,j=1,2)$ satisfy the Cord\'es condition  (cf. \cite{MPS00}):
\begin{equation}
\label{cordes}
\frac{ \sum_{i,j=1}^2 a_{ij}^2}{(\sum_{i=1}^2 a_{ii})^2} \le \frac{1}{1+
\epsilon},\qquad \text{in} \quad \Omega,
\end{equation}
for a positive number $\epsilon \in (0, 1]$.  This condition is
reasonable in $\mathbb{R}^2$ in the sense that
when the coefficient tensor $a(x)$ satisfies  the standard uniform 
ellipticity, i.e.,
there exist two positive numbers $\lambda_1$ and $\lambda_2$ such that
\begin{equation}
\label{ellipticity}
\lambda_1  \xi^\top \xi  \le    \xi^\top  a(x) \xi  \le \lambda_2
\xi^\top \xi, \quad \forall \xi \in
\mathbb{R}^2, x\in \Omega,
\end{equation}
then the Cord\'es condition holds true in $\mathbb{R}^2$ 
(cf. \cite{MPS00}). Indeed, 
when $c\equiv 0$, the existence of the unique strong solution
$u$ satisfying (\ref{regularity}) can be found in \cite{MPS00}
based on a contraction map argument. Recently, in \cite{SS13}, the
researchers weakened the assumption on the $C^2$ smooth boundary and  
use the well-known Lax-Milgram theorem to establish the 
weak solution  to (\ref{model}) with $c\equiv 0$ 
when $\Omega$ is convex with Lipschitz differentiable boundary 
$\partial\Omega$. In fact, as the solution $u\in H^2(\Omega)$, the
weak solution is the strong solution.  
It is easy to see that the proofs in \cite{MPS00} and \cite{SS13} can 
be extended to establish the existence and uniqueness of the strong 
solution in $H^2(\Omega)$ for the PDE in (\ref{model}) when 
$\Omega$ is convex as each convex domain automatically 
has Lipschitz differentiable boundary, i.e. $C^{1,1}$ boundary.  
Numerical approximations of $u$ using the discontinuous Galerkin (DG) 
method and the weak Galerkin (WG) method  
have been studied in \cite{SS13} and \cite{WW15} 
to solve (\ref{model}) with $c\equiv 0$. 
The convergence and convergence rates of these two numerical methods
were established together with  numerical evidence of convergence over 
convex and nonconvex domains. 

In this paper, our goal is to provide another efficient computational 
method for numerical solution of (\ref{model}). More precisely, 
we propose a bivariate spline method based on the minimization of 
the jumps of functions across edges and the boundary condition.  
Bivariate splines are discontinuous piecewise polynomial functions 
which are written in Bernstein-B\'ezier polynomial form 
(cf. \cite{LS07}) and the smoothness constraints 
across an edge $e$ of triangulation $\triangle$ are written in terms of 
the coefficients of polynomials over the two triangles 
sharing the edge $e$. In particular,  smoothness conditions 
for any smoothness $r\ge 0$ across
edges of any order are implemented in MATLAB which can be simply used.
This is an improvement over the internal penalties in the DG method. 
Bivariate splines have been used for numerical solutions of 
various types of PDE. See \cite{ALW06}, \cite{LW04}, \cite{HHL07},  
\cite{A15}, \cite{GLS15}, \cite{S15}, and etc.. They can be very 
convenient for numerical solutions of this type of PDE. See an 
extensive numerical evidence in \S 6.
  
Note that in \cite{SS13}, an hp-version discontinuous Galerkin finite
element method was proposed and analyzed. It was based on a variational 
form arising from testing the original PDE against the Laplacian of 
sufficiently smooth functions (e.g., twice differentiable functions 
in $L^2$). 
According to the researchers in \cite{SS13},  their method yielded an 
optimal order of convergence regarding to the mesh size $h$, i.e. 
$d-1$ for polynomial degree $d=2, 3, 4, 5$.  
We use the same $C^1$ test function for 
a PDE with non-differentiable coefficients and provide an evidence 
that the convergence rate of $|u-S_u|_{H^2(\Omega)}$ using 
bivariate spline method  is also 
$d-1$ for $d=5, 6, 7, 8$  when $c\equiv 0$. In addition, we present 
numerical convergence of $u-S_u$ in $L^2$ norm and $H^1$ 
semi-norm which can be better than $d-1$ for various $d$. When $c\not
=0$, we have the same convergence behavior. In particular,  
the convergence rate of $|u-S_u|_{H^2(\Omega)}$ is still $d-1$. 
 
Also, it is worthy pointing out that the 
researchers in \cite{WW15} proposed a  primal-dual weak Galerkin finite 
element method for this type of PDE (i.e. $c\equiv 0$), yielding an 
optimal order error estimate  in a discrete $H^2$-norm 
for the primal variable and in the $L^2$-norm for the dual 
variable,  as well as  a convergence theory for the primal variable 
in the $H^1$- and $L^2$-norms. 
Although the numerical method implemented in 
\cite{WW15} works for finite element partitions consisting of arbitrary 
polygons or polyhedra, only polynomials 
of degree 2 over triangulation were implemented and 
used for numerical solution. 
The flexibility of using bivariate splines of various degrees make our 
method more convenient to increase 
the accuracy of  solutions.  More accurate numerical solutions than 
the ones in \cite{WW15} are presented  in this paper to demonstrate 
the advantage of our bivariate spline method.  

The theoretical study in this paper uses 
the same assumptions as in \cite{WW15}
and extend  their arguments in \cite{WW15}  to analyze the PDE in 
(\ref{model}) with nonzero $c$ and  
the convergence of the bivariate spline method. 
Our analysis in this paper is significantly different from that 
for discontinuous Galerkin method in \cite{SS13}. 
The paper is organized as follows: We first start with an explanation of
the primal-dual discontinuous Galerkin method to solve (\ref{model}) in 
the next section. 
Mainly, we establish some basic properties such as the existence, 
uniqueness, stability of the method in Section~\ref{sec:properties}. 
Then in Section~\ref{sec:EE} we present an error analysis of the 
numerical solution. Next we reformulate the primal-dual discontinuous 
Galerkin algorithm based on the bivariate 
spline functions and  explain our implementation based on bivariate 
splines explained in \cite{ALW06}. 
Extensive numerical results are reported in Section~\ref{sec:nr}. We 
start with a PDE with smooth coefficients
and test on a smooth solution to demonstrate that the bivariate spline 
method works very well. For comparison
purpose, we use the PDE in (\ref{model}) with $c\equiv 0$. Then we 
apply the bivariate spline method to 
numerically solve a few PDE with non-differentiable coefficients and 
non-smooth solutions as well as $c\not=0$. Our method is able
to approximate the nonsmooth solution very well. 
Therefore, the bivariate spline method is effective and efficient. 

\section{A Primal-Dual Discontinuous Galerkin Scheme}  
Our model problem seeks for a function $u\in H^2(\Omega)$ satisfying
$u|_{\partial \Omega}=0$ and
\begin{equation}\label{weak}
(\sum_{i,j=1}^2 a_{ij}\partial_{ij}^2 u+cu,w) =(f,w), \qquad \forall
w\in L^2(\Omega).
\end{equation}

Let ${\cal T}_h$ be a polygonal finite
element partition of the domain $\Omega \subset \mathbb R^2$. Denote by ${\mathcal
E}_h$ the set of all edges in ${\cal T}_h$ and ${\mathcal
E}_h^0={\mathcal E}_h \setminus \partial\Omega$ the set of all
interior edges. Assume that ${\cal T}_h$ satisfies the shape
regularity conditions described in \cite{BS94, wy3655}.
Denote by $h_T$ the diameter  of $T\in {\cal T}_h$ and $h=\max_{T\in
{\cal T}_h}h_T$ the mesh size of the partition ${\cal T}_h$. Let $k \geq 0$ be an integer. Let $P_k(T)$ be  the space of polynomials of degree no more than $k$ on the element $T\in {\cal T}_h$.

For any given integer $k \geq 2$,  we define the finite element spaces composed of piecewise polynomials of degree $k$ and $k-2$, respectively; i.e., 
$$
X_h=\{u: u|_T\in P_k(T), \quad \forall T\in {\cal T}_h \},  
$$
$$
M_h=\{u: u|_T\in P_{k-2}(T), \quad  \forall T\in {\cal T}_h \}. 
$$
 
Denote by $\jump{v}$ the jump of $v$ on an 
edge  $e \in {\mathcal E}_h$; i.e.,  
\begin{equation}\label{jump}
 \begin{split}
   \jump{v}=\left\{
\begin{array}{ll}
v|_{T_1}-v|_{T_2}, & e=(\partial T_1\cap \partial T_2)\subset
{\mathcal E}_h^0, \\
v, &  e\subset \partial\Omega,
 \end{array}
\right.
 \end{split}
\end{equation}
where $v|_{T_i}$ denotes the value of $v$ as seen from the element $T_i$,
$i=1,2$. The order of $T_1$ and $T_2$ is non-essential in
(\ref{jump}) as long as the difference is taken in a consistent way
in all the formulas. Analogously, one may define the jump of the
gradient of $u$ on an edge $e \in{\mathcal E}_h$, denoted by  $\jump{\nabla u}$.

For  any $v\in X_h$,   
 the  quadratic functional $J(v) $ is given by 
\begin{equation}\label{EQ:functional}
\begin{split}
J(v) = & \frac12 \sum_{e\in \E_h} h_T^{-3} \langle \jump{v}, \jump{v}
\rangle_e + \frac12 \sum_{e\in \E_h^0} h_T^{-1} \langle \jump{\nabla
v}, \jump{\nabla v} \rangle_e.
\end{split}
\end{equation}
It is clear that $J(v)=0$ if and only if $v\in C^1(\Omega)\cap X_h$
with  the homogeneous Dirichlet boundary data $v = 0$ on $\partial\Omega$.

We introduce a bilinear form
\begin{equation}\label{bilinearform2}
b_h(v, q)= \sum_{T\in\T_h} ( \sum_{i,j=1}^2 a_{ij}\partial_{ij}^2 v+cv,
q)_T, \quad  \forall v\in X_h, \ \forall q\in M_h.
\end{equation} 

The numerical solution  of  the model problem
(\ref{model})-(\ref{model:bdrycond}) can be characterized a constrained minimization problem as follows: {\em
Find $u_h\in X_h$ such that
\begin{equation}\label{DG:ConstrainedOpt}
u_h=\argmin_{v\in X_{h}, \ b_h(v, q)=(f,q),\ \forall q\in M_{h}}
J(v).
\end{equation}
}

 By introducing the following bilinear form
\begin{equation}\label{bilinearform}
s_h(u, v) = \sum_{e\in \E_h} h_T^{-3} \langle \jump{u}, \jump{v}
\rangle_e + \sum_{e\in \E_h^0} h_T^{-1} \langle \jump{\nabla v},
\jump{\nabla v} \rangle_e,\quad \forall u, v\in X_h, 
\end{equation}
the constrained minimization problem (\ref{DG:ConstrainedOpt}) has
an Euler-Lagrange formulation that gives rise to a system of linear
equations by taking the Fr\'echet derivative. The 
Euler-Lagrange formulation for the constrained minimization
algorithm (\ref{DG:ConstrainedOpt}) gives   the following numerical scheme.
\begin{algorithm}\emph{(Primal-Dual Discontinuous Galerkin FEM)}
\label{ALG:primal-dual-dg-fem}A numerical approximation of the
second order elliptic problem (\ref{model})-(\ref{model:bdrycond}) seeks to
find $(u_h;\lambda_h)\in X_{h} \times M_{h}$ satisfying
\begin{eqnarray}\label{EQ:PD-DG:001}
s_h(u_h,v)+b_h(v, \lambda_h)&=& 0,\ \ \ \ \qquad \forall v\in X_{h},\\
b_h(u_h, q)&=&(f,q),\qquad \forall q\in M_{h}.\label{EQ:PD-DG:002}
\end{eqnarray}
\end{algorithm}

\section{Existence, Uniqueness  and Stability\label{sec:properties}}
In this section, we will derive the existence, uniqueness, and
stability for the solution $(u_h;\lambda_h)$  of the  primal-dual
discontinuous Galerkin scheme
(\ref{EQ:PD-DG:001})-(\ref{EQ:PD-DG:002}).

For each element $T\in {\cal T}_h$, let $B_T$ be  the largest disk inside of $T$ centered at $c_0$  with    radius   $r$
and $F_{k,B_T}(f)$ be the averaged Taylor polynomial of degree $k$
for $f\in L^1(T)$ (see page 4 of \cite{LS07} for details). Note that
the averaged Taylor polynomial $F_{k,B_T}(f)$ satisfies (cf. Lemma
1.5 in \cite{LS07})
\begin{equation} \label{LSp41}
{\partial_{ij}^2 F_{k,B_T}(f)} = F_{k-2,B_T} ({\partial_{ij}^2f})
\end{equation}
if $\partial_{ij} f\in L^1(T)$. Let $P_{X_h}(f)$ and $P_{M_h}(f)$ be
interpolations/projections of $f$ onto the spaces $X_h$ and $M_h$ defined by
$P_{X_h}(f)|_T= F_{k,B_T}(f)$ and $P_{M_h}(f)|_T= F_{k-2,B_T}(f)$ on
each element $T\in {\cal T}_h$, respectively. Using (\ref{LSp41}) gives rise to   \begin{equation}
\label{LSp4} {\partial_{ij}^2 P_{X_h}(f) } =
P_{M_h}({\partial_{ij}^2 f}),
\end{equation}
on each element  $T\in {\cal T}_h$,

 \begin{lemma} \cite{LS07}
\label{newproperty2-ww} The interpolant operators $P_{X_h}$ and  $P_{M_h}$ are
bounded in $L^2(\Omega)$. In other words, for any $f\in L^2(\Omega)$
we have
\begin{equation}\label{part1}
  \|P_{X_h}(f)\| \le C  \|f\|, 
\end{equation}
\begin{equation}\label{part2}
  \|P_{M_h}(f)\| \le C  \|f\|,
\end{equation}
where $C$ is a constant depending only on the shape parameter
$\theta_{{\cal T}_h}=\max_{T\in {\cal T}_h} \frac{h_T}{\rho_T}$,
$\rho_T$ is the radius of the largest inscribed circle of $T$.
\end{lemma}

Recall that ${\cal T}_h$ is a shape-regular finite element partition of the domain $\Omega$. For any $T\in {\cal T}_h$ and $\phi\in H^1(T)$, the following trace inequality holds true:
\begin{equation}\label{tracein}
 \|\phi\|^2_{\partial T} \leq C(h_T^{-1}\|\phi\|_T^2+h_T \|\nabla \phi\|_T^2).
\end{equation}

Deonte by $Q_{k-2}$ the $L^2$ projection onto the finite element space $M_h$. 
We introduce a semi-norm in the finite element space $X_h$, denoted by  $\3bar\cdot\3bar$; i.e.,
\begin{equation}
\label{newnorm} \3bar v \3bar  =\Big( \sum_{T\in {\cal T}_h}
\| Q_{k-2}( \sum_{i,j=1}^2a_{ij}\partial^2_{ij} v+cv)\|_T^2 +
s_h(v,v)\Big)^{\frac{1}{2}}, \qquad  v\in X_h.
\end{equation}

   The following result shows that $\3bar\cdot\3bar$ defined in (\ref{newnorm}) is indeed a norm on $X_h$
when the meshsize $h$ is sufficiently small.
\begin{lemma}\label{WWideas} Assume that the $H^2$ regularity
(\ref{regularity}) holds true for the model problem
(\ref{model})-(\ref{model:bdrycond}), and that the coefficient
tensor $a(x)=\{a_{ij}(x)\}_{2 \times 2}$ and  $c(x)$ are uniformly piecewise
continuous in $\Omega$ with respect to the finite element partition
${\cal T}_h$. Then, there exists an $h_0>0$ such that $\3bar
\cdot\3bar$ in (\ref{newnorm}) defines a norm on $X_h$ when the
meshsize $h$ is sufficiently small such that $h \le h_0$.
\end{lemma}

\begin{proof}
It suffices to verify the positivity property for $\3bar\cdot\3bar$.
To this end, note that for any $v\in X_h$ satisfying $\3bar v
\3bar=0$ we have $s_h(v, v)=0$. It follows that $\jump{v}=0$ on each
edge $e\in {\mathcal E}_h$ and $\jump{\nabla u}=0$ on each interior
edge $e\in {\mathcal E}_h^0$. Hence, $v\in C^1(\Omega)$ and $v=0$ on
$\partial\Omega$. In addition, on each element $T\in {\cal T}_h$, we
have
$$
Q_{k-2}( \sum_{i,j=1}^2a_{ij}\partial^2_{ij} v+cv)=0.
$$
Thus,
\begin{align*}
 \sum_{i,j=1}^2 a_{ij}\partial^2_{ij} v+cv  & = (I-Q_{k-2})
( \sum_{i,j=1}^2a_{ij}\partial^2_{ij} v+cv):=F.
\end{align*}
Using the $H^2$-regularity assumption (\ref{regularity}), there
exists a constant $C$ such that
\begin{equation}\label{h2}
 \|v\|_2\leq C\|F\|.
\end{equation}
Note that $a_{ij}(x)$ and $c(x)$ are uniformly piecewise continuous in $\Omega$
with respect to the finite element partition ${\cal T}_h$. Let
$\bar{a}_{ij}$ and $\bar{c}$  be the average of $a_{ij}$  and $c$ on each element $T\in
{\cal T}_h$. Then, for any $\varepsilon>0$, there exists a $h_0>0$
such that
$$
\|a_{ij}-\bar{a}_{ij}\|_{L^{\infty}(\Omega)}\leq \varepsilon,\qquad \|c-\bar{c} \|_{L^{\infty}(\Omega)}\leq \varepsilon,
$$ 
if the meshsize $h$ is  sufficiently small such that $h\le h_0$. Denote by $\bar{c}$ and $\bar{v}$ the average of $c$ and $v$ on each element $T\in {\cal T}_h$, respectively.  It follows from the
linearity of the projection $Q_{k-2}$ that
\begin{equation*}
\begin{split}
\|F\|  \le &\sum_{i,j=1}^2 |a_{ij}-\bar{a}_{ij}| \|\partial^2_{ij}v\|
+\sum_{i,j=1}^2 \|Q_{k-2}((a_{ij}-
\bar{a}_{ij})\partial^2_{ij}v)\|\\
&+  \|(I-Q_{k-2})(cv-\bar{c}\bar{v})\| \\
\leq & C \varepsilon \|v\|_2+ \|cv-\bar{c}\bar{v}\| \\ 
\leq & C \varepsilon \|v\|_2+ \|(c-\bar{c}) v+\bar{c}(v-\bar{v})\| \\ 
\leq & C \varepsilon \|v\|_2+C \varepsilon \|v\|+ Ch\|v\|_1\\
\leq & C \varepsilon \|v\|_2+Ch\|v\|_2,
\end{split}
\end{equation*}
where we have used the boundedness of the $L^2$ projection $Q_{k-2}$,  which, combined with (\ref{h2}),  gives
$$
\|v\|_2\leq C (\varepsilon +h)\|v\|_2.
$$
This yields that $v=0$ as long as $\varepsilon$ is sufficiently
small such that $C \varepsilon < 1$, which can be
easily achieved  by adjusting the parameter $h_0$. This completes
the proof of the lemma.
\end{proof}
 
We are now in a position to establish an {\em inf-sup} condition for
the bilinear form $b_h(\cdot,\cdot)$.

\begin{lemma}(inf-sup  condition)
\label{infsup2} Under the assumptions of Lemma \ref{WWideas},   for any $q \in M_h$, there exists a $v_q
\in X_h$ such that
\begin{eqnarray}
\label{critical} b_h(v_q, q) & \ge &  \beta \|q\|^2,\\
\3bar v_q \3bar &\le & C \|q\|, \label{critical.002}
\end{eqnarray}
provided that the meshsize $h$ is sufficiently small.
\end{lemma}

\begin{proof}
Consider an auxiliary problem that seeks $w\in H^2(\Omega)\cap
H_0^1(\Omega)$ satisfying
\begin{equation}
\label{PDEp} \sum_{i,j=1}^2 a_{ij} \partial^2_{ij} w +cw= q,\qquad
\text{in}\ \Omega.
\end{equation}
From the regularity assumption (\ref{regularity}), it is easy to know that  the problem
(\ref{PDEp}) has one and only one solution, and furthermore, the
solution satisfies  the $H^2$
regularity property; i.e.,
\begin{equation}
\label{reg} \|w\|_2\le  C \|q\|.
\end{equation}
By letting $v_q= P_{X_h}(w)$, from (\ref{LSp4}) we obtain
\begin{align*}
\partial^2_{ij} v_q  = \partial^2_{ij} P_{X_h}(w)
 =  P_{M_h}(\partial^2_{ij} w).
\end{align*}
Letting ${\bar a}_{ij}$ be  the average of $a_{ij}$ over $T\in {\cal T}_h$,
we arrive at
\begin{align*}
&\sum_{i,j=1}^2 a_{ij} \partial^2_{ij} v_q+c v_q  \\
=& \sum_{i,j=1}^2 \{(a_{ij} - {\bar a}_{ij})
P_{M_h}(\partial^2_{ij} w) + P_{M_h}( {\bar a}_{ij} \partial^2_{ij} w) \}+ (c-\bar{c})P_{X_h}(w)+P_{X_h}(\bar{c}  w) \\
  =& \sum_{i,j=1}^2 \{(a_{ij} - {\bar a}_{ij}) P_{M_h}(\partial^2_{ij} w) + P_{M_h}(
({\bar a}_{ij}-a_{ij}) \partial^2_{ij} w)
+P_{M_h}(a_{ij}\partial^2_{ij} w)\}\\
&+ (c-\bar{c})P_{X_h}(w)+P_{X_h}((\bar{c}-c)  w)+P_{X_h}(cw)\\
 =& \sum_{i,j=1}^2 \{(a_{ij} - {\bar a}_{ij}) P_{M_h}(\partial^2_{ij} w) + P_{M_h}(
({\bar a}_{ij}-a_{ij}) \partial^2_{ij} w)\}+ (c-\bar{c})P_{X_h}(w)\\
& +P_{X_h}((\bar{c}-c)  w) +P_{M_h}( \sum_{i,j=1}^2 a_{ij}\partial^2_{ij} w+cw) +P_{X_h}(cw)-P_{M_h}(cw)\\
 =& E_T+ q+P_{X_h}(cw)-P_{M_h}(cw). 
\end{align*}
where we have used (\ref{PDEp}) and $P_{M_h} q =q$. Here, $E_T=  \sum_{i,j=1}^2 \{(a_{ij} - {\bar a}_{ij}) P_{M_h}(\partial^2_{ij} w) + P_{M_h}(
({\bar a}_{ij}-a_{ij}) \partial^2_{ij} w)\}+ (c-\bar{c})P_{X_h}(w) 
  +P_{X_h}((\bar{c}-c)  w)$.  

With the above chosen $v_q$ as $P_{X_h}(w)$, we have
\begin{equation}\label{bf}
 \begin{split}
  b_h(v_q, q)& = \sum_{T\in \T_h} ( \sum_{i,j=1}^2 a_{ij}
\partial^2_{ij} P_{X_h}(w)+cP_{X_h}(w), q)_T \\
& = \sum_{T\in \T_h} (E_T, q)_T + \|q\|^2+\sum_{T\in \T_h}(P_{X_h}(cw)-P_{M_h}(cw), q)_T.
 \end{split}
\end{equation}
Note that the coefficient tensor $a(x)=\{  a_{ij}\}_{2\times 2}$ and $c(x)$ are
uniformly piecewise continuous over ${\cal T}_{h}$. Thus, for any
given suifficiently small  $\varepsilon>0$, we have $\|a_{ij}- {\bar
a}_{ij}\|_{L^\infty(\Omega)} \le  \varepsilon$ and $\|c- {\bar
c}\|_{L^\infty(\Omega)} \le  \varepsilon$ for sufficiently
small meshsize $h$. It then follows from the Cauchy-Schwarz
inequality, (\ref{part1}) - (\ref{part2}), and the $H^2$
regularity property (\ref{reg}) that
\begin{align*}
&\left| \sum_{T\in {\cal T}_h} ( E_T,q)_T \right|\\ \leq &
C\varepsilon \Big(\sum_{T\in {\cal T}_h} \sum_{i,j=1}^2\|P_{M_h}
(\partial^2_{ij}w)\|_{T}^2\Big)^{\frac{1}{2}} \Big(\sum_{T\in {\cal
T}_h}\|q\|_{T}^2\Big)^{\frac{1}{2}}+C\varepsilon \Big(\sum_{T\in {\cal T}_h} \sum_{i,j=1}^2\| 
 \partial^2_{ij}w \|_{T}^2\Big)^{\frac{1}{2}} \Big(\sum_{T\in {\cal
T}_h}\|q\|_{T}^2\Big)^{\frac{1}{2}}\\
&+C\varepsilon \Big(\sum_{T\in {\cal T}_h}  \|P_{X_h}
w\|_{T}^2\Big)^{\frac{1}{2}} \Big(\sum_{T\in {\cal
T}_h}\|q\|_{T}^2\Big)^{\frac{1}{2}}+C\varepsilon \Big(\sum_{T\in {\cal T}_h}  \| 
w\|_{T}^2\Big)^{\frac{1}{2}} \Big(\sum_{T\in {\cal
T}_h}\|q\|_{T}^2\Big)^{\frac{1}{2}}\\
 \leq &   C\varepsilon
\Big(\sum_{T\in {\cal T}_h} \sum_{i,j=1}^2\|
\partial^2_{ij}w \|_{T}^2\Big)^{\frac{1}{2}}
 \|q\|+ C\varepsilon
\Big(\sum_{T\in {\cal T}_h} \| w \|_{T}^2\Big)^{\frac{1}{2}}
 \|q\| \cr
\le &  C\varepsilon \|w\|_2\|q\| \leq   C \varepsilon \|q\|^2,
\end{align*}
and 
\begin{align*}
&\Big|\sum_{T\in {\cal T}_h} (P_{X_h}(cw)-P_{M_h}(cw), q)_T\Big| \\
\leq &\Big(\sum_{T\in {\cal T}_h} \|P_{X_h}(cw- \bar{c}\bar{w})-P_{M_h}(cw- \bar{c}\bar{w})\|^2_T\Big)^{\frac{1}{2}}\Big(\sum_{T\in {\cal T}_h}\|q\|^2_T\Big)^{\frac{1}{2}}\\
\leq &\Big(\sum_{T\in {\cal T}_h} \|cw- \bar{c}\bar{w}\|^2_T\Big)^{\frac{1}{2}}\Big(\sum_{T\in {\cal T}_h}\|q\|^2_T\Big)^{\frac{1}{2}}\\
\leq &\Big(\sum_{T\in {\cal T}_h} \|(c- \bar{c})w+\bar{c}(w-\bar{w})\|^2_T\Big)^{\frac{1}{2}}  \|q\| \\
\leq & (C\varepsilon \|w\|+Ch\| w\|_1)\|q\|\\
\leq & C (\varepsilon +h)\|w\|_2\|q\|\\
\leq & C  (\varepsilon +h)\|q\|^2,
\end{align*}
where $\bar{c}$ and $\bar{w}$ are the average of $c$ and $w$ on each element $T\in {\cal T}_h$, respectively,
  $C$ is a generic constant independent of ${\cal T}_h$.
Substituting the above estimate into (\ref{bf}) yields
$$
b_h(v_q,q)\ge (1- C (2\varepsilon+h)) \|q\|^2,
$$
which leads to the estimate (\ref{critical}) when the meshsize $h$
is sufficiently small.

It remains to derive the estimate (\ref{critical.002}). To this end,
recall that
\begin{equation}\label{1anorm5}
\3bar v_q \3bar^2 = \sum_{T\in{\cal T}_h}\| 
Q_{k-2}(\sum_{i,j=1}^2a_{ij}\partial^2_{ij} v_q+cv_q)\|^2_{T} + s_h(v_q, v_q).
\end{equation}
Letting $v_q= P_{X_h}(w)$, the first term on the
right-hand side of (\ref{1anorm5}) can be bounded by using (\ref{LSp4}),
(\ref{part1}) and (\ref{reg}) as follows:
\begin{equation}\label{1anorm5a}
\begin{split}
\sum_{T\in{\cal T}_h}\| Q_{k-2}(\sum_{i,j=1}^2a_{ij}\partial^2_{ij}
v_q+cv_q) \|^2_{T} & = \sum_{T\in{\cal T}_h}\|Q_{k-2}(\sum_{i,j=1}^2 a_{ij}\partial^2_{ij}
P_{X_h}(w)+cP_{X_h}(w))\|^2_{T} \\
 & \le \sum_{T\in{\cal T}_h}  \sum_{i,j=1}^2
\|a_{ij}P_{M_h}(\partial^2_{ij} w)\|_T+\|cP_{X_h}(w)\|^2_{T}  \\
& \le C \sum_{i,j=1}^2 \|a_{ij}\|^2_{L^{\infty}(\Omega)}  
\sum_{T\in{\cal T}_h} \| \partial^2_{ij} w\|_T^2+C  \|c\|^2_{L^{\infty}(\Omega)}  
\sum_{T\in{\cal T}_h} \|   w\|_T^2   \\
& \le C\|q\|^2.
\end{split}
\end{equation}

As to the term $s_h(v_q, v_q)$ in (\ref{1anorm5}), note that it is
defined by (\ref{bilinearform}) using the jump of $v_q$ on each edge
$e\in \E_h$ plus the jump of $\nabla v_q$ on each interior edge
$e\in \E_h^0$. For an interior edge $e\in\E_h^0$ shared by two
elements $T_1$ and $T_2$, we have
\begin{eqnarray*}
\jump{v_q}|_e &=& v_q|_{T_1\cap e} - v_q|_{T_2\cap e} \\
& = &  P_{X_h}(w) |_{T_1\cap e} - P_{X_h}(w) |_{T_2\cap e} \\
& = &  (P_{X_h}(w) |_{T_1\cap e} -w|_e) + (w|_e - P_{X_h}(w)
|_{T_2\cap e}).
\end{eqnarray*}
It follows that
\begin{equation}\label{EQ:Sep:16:001}
\langle \jump{v_q}, \jump{v_q}\rangle_e \leq 2
 \|P_{X_h}(w) |_{T_1\cap e} -w|_e\|_e^2 + 2 \|P_{X_h}(w) |_{T_2\cap e}
 -w|_e\|_e^2.
\end{equation}
Using the trace inequality (\ref{tracein}), we have
$$
\|P_{X_h}(w) |_{T_1\cap e} -w|_e\|_e^2 \leq C
h_T^{-1}\|P_{X_h}(w)-w\|^2_{T_1} + C h_T
\|\nabla(P_{X_h}(w)-w)\|^2_{T_1}.
$$
Analogously, the following holds true
$$
\|P_{X_h}(w) |_{T_2\cap e} -w|_e\|_e^2 \leq C
h_T^{-1}\|P_{X_h}(w)-w\|^2_{T_2} + C h_T
\|\nabla(P_{X_h}(w)-w)\|^2_{T_2}.
$$
Substituting the last two inequalities into (\ref{EQ:Sep:16:001})
yields
\begin{equation}\label{EQ:Sept:16:002}
\langle \jump{v_q}, \jump{v_q}\rangle_e \leq C\sum_{i=1}^2\left(
h_T^{-1}\|P_{X_h}(w)-w\|^2_{T_i} + C h_T
\|\nabla(P_{X_h}(w)-w)\|^2_{T_i} \right).
\end{equation}
For boundary edge $e\subset  \partial \Omega$, from $w|_{e\subset  \partial \Omega}=0$ we have
\begin{eqnarray*}
\jump{v_q}|_e &=& v_q|_{e}= P_{X_h}(w)|_{e} - w|_{e}.
\end{eqnarray*}
Thus, the estimate (\ref{EQ:Sept:16:002}) remains to hold true.
Summing (\ref{EQ:Sept:16:002}) over all the edges yields
\begin{equation}\label{EQ:Sept:16:003}
\begin{split}
\sum_{e\in\E_h} h_T^{-3} \langle \jump{v_q}, \jump{v_q}\rangle_e &
\leq C\sum_{T\in\T_h} \left(h_T^{-4}\|P_{X_h}(w)-w\|^2_{T} +
C h_T^{-2}\|\nabla(P_{X_h}(w)-w)\|^2_{T} \right)\\
& \le C \|w\|_2^2\end{split}
\end{equation}
where we have used the estimate (\ref{errorXh}) with $m=1$ and
$s=0,1$ in the last inequality. Combining (\ref{EQ:Sept:16:003})
with the regularity estimate (\ref{reg}) gives rise to
\begin{equation}\label{EQ:Sept:16:00488}
\sum_{e\in\E_h} h_T^{-3} \langle \jump{v_q}, \jump{v_q}\rangle_e \leq
C \|q\|^2.
\end{equation}
A similar argument can be applied to yield the following estimate
\begin{equation}\label{EQ:Sept:16:004}
\sum_{e\in\E_h^0} h_T^{-1} \langle \jump{\nabla v_q}, \jump{\nabla
v_q}\rangle_e \leq C \|q\|^2.
\end{equation}
We emphasize that the summation in (\ref{EQ:Sept:16:004}) is taken
over all the interior edges so that no boundary value for $\nabla w$
is needed in the derivation of the estimate (\ref{EQ:Sept:16:004}).
Combining (\ref{EQ:Sept:16:00488}) and (\ref{EQ:Sept:16:004}) with
$s_h(v_q,v_q)$ yields
$$
s_h(v_q,v_q) \leq C \|q\|^2,
$$
which, together with (\ref{1anorm5a}), completes the derivation of
the estimate (\ref{critical.002}).
\end{proof}

\begin{lemma}\label{boundlem}
(Boundedness) The following inequalities hold true:
\begin{equation*}
 \begin{split}
  |s_h(u,v)|\leq & \3bar u\3bar \3bar v\3bar, \qquad \forall u,v\in X_h,\\
|b_h(v,q)|\leq & C \3bar v\3bar \|q\|, \qquad \forall v\in X_h, q\in
M_h.
 \end{split}
\end{equation*}
 \end{lemma}
\begin{proof}
It follows from the definition of $s_h(\cdot,\cdot)$, $\3bar \cdot
\3bar$   and Cauchy-Schwarz inequality that  for any $u, v\in X_h$,
we have
\begin{equation*}
\begin{split}
  |s_h(u,v)|   = & \Big| \sum_{e\in \E_h} h_T^{-3} \langle \jump{u}, \jump{v}
\rangle_e + \sum_{e\in \E_h^0} h_T^{-1} \langle  \jump{\nabla u}, \jump{\nabla v} \rangle_e\Big|\\
 \leq & \Big( \sum_{e\in \E_h} h_T^{-3} \langle \jump{u}, \jump{u}
\rangle_e \Big) ^{\frac{1}{2}}  \Big( \sum_{e\in \E_h} h_T^{-3} \langle \jump{v}, \jump{v}
\rangle_e \Big) ^{\frac{1}{2}}\\ &+ \Big( \sum_{e\in \E_h^0} h_T^{-1} \langle \jump{ \nabla u}, \jump{\nabla u}
\rangle_e \Big) ^{\frac{1}{2}}  \Big( \sum_{e\in \E_h^0} h_T^{-1} \langle \jump{\nabla v}, \jump{\nabla v}
\rangle_e \Big) ^{\frac{1}{2}}\\
 \leq  & s_h(u, u)^{\frac{1}{2}} s_h(v, v)^{\frac{1}{2}} \\
 \le  & \3bar
u\3bar \3bar v\3bar.
\end{split}
\end{equation*}

Next from the definition of $b_h(\cdot,\cdot)$, $\3bar \cdot \3bar$,
and Cauchy-Schwarz inequality that for any $v\in X_h$, $q\in M_h$,
we have
\begin{equation*}
 \begin{split}
|b_h(v,q)| = &\Big|\sum_{T\in\T_h} (\sum_{i,j=1}^2
  a_{ij}
\partial_{ij}^2 v+cv, q)_T\Big|\\
=&\Big|\sum_{T\in\T_h}  ( 
Q_{k-2}(\sum_{i,j=1}^2a_{ij}
\partial_{ij}^2 v+cv), q)_T\Big|\\
\leq & (\sum_{T\in\T_h}\|Q_{k-2}(\sum_{i,j=1}^2a_{ij}
\partial_{ij}^2 v+cv)\|_T^2)^{\frac{1}{2}}(\sum_{T\in\T_h}  \|q
\|_T^2)^{\frac{1}{2}}\\
 \leq & \3bar v\3bar \|q\|.
 \end{split}
\end{equation*}
These complete the proof.
\end{proof}

Define  the subspace of $X_h$ as follows:
\begin{equation*}
\Xi_h =\{ {  v}\in X_h: \quad b_h(v, q)=0, \quad \forall { q}\in M_h \}.
\end{equation*}

\begin{lemma}\label{coerlemm}
(Coercivity)  
There exists a constant $ \alpha$, such that
$$
s_h(v, v)\geq  \alpha\3bar v\3bar ^2, \qquad \forall v\in
\Xi_h.
$$
\end{lemma}
\begin{proof}
For any $v\in  \Xi_h$, we   have
$$
b_h(v, q)=0,\qquad \forall q\in M_h.
$$
It follows from the definition of $b(\cdot,\cdot)$ in
(\ref{bilinearform2})  that
$$
0=b_h(v, q)=\sum_{T\in {\cal T}_h}( \sum_{i,j=1}^2a_{ij}
\partial_{ij}^2 v+cv,q)_T= \sum_{T\in {\cal T}_h}( Q_{k-2}(\sum_{i,j=1}^2a_{ij}
\partial_{ij}^2 v+cv),q)_T,
$$
which yields
$$
Q_{k-2}(\sum_{i,j=1}^2a_{ij}
\partial_{ij}^2 v+cv)  =0,
$$
on each $T\in {\cal T}_h$ by letting
$q=Q_{k-2}(\sum_{i,j=1}^2a_{ij}
\partial_{ij}^2 v+cv)$. This implies $s_h(v,v)=\3bar
v\3bar^2$, which completes the proof with $\alpha=1$.
\end{proof}

Using the abstract theory for the saddle-point problem developed by Babuska \cite{B73}
and Brezzi \cite{B74}, we arrive at the following theorem based on  
  Lemmas \ref{infsup2} - \ref{coerlemm}.

 \begin{theorem}\label{main2} 
 The primal-dual discontinuous Galerkin finite element method (\ref{EQ:PD-DG:001})-(\ref{EQ:PD-DG:002}) has a unique
solution $(u_h; \lambda_h)\in X_h\times M_h$, provided that the meshsize $h<h_0$ holds true for a sufficiently small but fixed parameter $h_0>0$.  Moreover, there exists a constant C such that the
solution $(u_h; \lambda_h)$ satisfies 
\begin{equation}\label{bound-jw} \3bar u_h\3bar +\|\lambda_h\|  \le C\|f\| .
\end{equation}
\end{theorem}

\section{Error Estimates\label{sec:EE}}
Let $(u_h;\lambda_h)\in X_h\times M_h$
be the approximate solution of the model problem (\ref{model}) arsing from primal-dual discontinuous Galerkin finite element method  (\ref{EQ:PD-DG:001})-(\ref{EQ:PD-DG:002}). 
Note that $\lambda=0$ is the exact solution of the trival dual problem $b_h(v, \lambda)=0$ for all $v\in H^2(\Omega)$.  Define the errors functions   by
$$
e_h = u_h - P_{X_h} u,\qquad \epsilon_h = \lambda_h - P_{M_h}\lambda.
$$

\begin{lemma}
 The error functions $e_h$ and $\epsilon_h$ satisfy the following  equations:
\begin{align} \label{err1}
  s_h(e_h,v)+b_h(v,\epsilon_h)=&  -s_h(P_{X_h}u,v),\qquad \forall v\in X_h,\\
b_h(e_h,p)=&l_u(p),\qquad\quad \qquad\quad  \forall  p\in M_h,\label{err2}
\end{align}
where $l_u(p)=\sum_{T\in{\cal T}_h}\sum_{i,j=1}^2 (a_{ij}(I-P_{M_h})\partial_{ij}^2u,p)_T+\sum_{T\in{\cal T}_h} (c( I-P_{X_h})u,p)_T$.
\end{lemma}

\begin{proof}
 By subtracting $s_h(P_{X_h}u,v)$ from both sides of (\ref{EQ:PD-DG:001}), we obtain
$$
s_h(u_h - P_{X_h} u,v)+b_h(v, \lambda_h-0)= -s_h(P_{X_h}u,v),\qquad \forall v\in X_h,
$$
which completes the proof of (\ref{err1}). 

Substracting $b_h(P_{X_h}u,p)$ from both sides of (\ref{EQ:PD-DG:002}), it follows from (\ref{LSp4}) and (\ref{model}) that
\begin{equation*}
 \begin{split}
&  b_h(u_h,p)-b_h(P_{X_h}u,p)\\
  =&(f,p)-b_h(P_{X_h}u,p)\\
=&(f,p)-\sum_{T\in{\cal T}_h} \sum_{i,j=1}^2(a_{ij} \partial_{ij}^2 (P_{X_h}u)+cP_{X_h}u,p)_T\\
=&( f,p)-\sum_{T\in{\cal T}_h} \sum_{i,j=1}^2(a_{ij} P_{M_h}(\partial_{ij}^2 u)+cP_{X_h}u,p)_T\\
 =&(f,p)- \sum_{T\in{\cal T}_h}  (\sum_{i,j=1}^2 a_{ij}\partial_{ij}^2  u+cu,p)_T-\sum_{T\in{\cal T}_h} \sum_{i,j=1}^2(a_{ij}( P_{M_h}-I)\partial_{ij}^2  u,p)_T-\sum_{T\in{\cal T}_h} (c( P_{X_h}-I)u,p)_T\\
=&( f,p)-  ( f,p)-\sum_{T\in{\cal T}_h} \sum_{i,j=1}^2(a_{ij}( P_{M_h}-I)\partial_{ij}^2  u,p)_T-\sum_{T\in{\cal T}_h} (c( P_{X_h}-I)u,p)_T\\
= &\sum_{T\in{\cal T}_h} \sum_{i,j=1}^2(a_{ij}( I-P_{M_h} )\partial_{ij}^2  u,p)_T+\sum_{T\in{\cal T}_h} (c( I-P_{X_h})u,p)_T,\\
\end{split}
\end{equation*}
which completes the proof of (\ref{err2}). 
\end{proof}

The equations (\ref{err1}) and (\ref{err2}) are called error equations for the primal-dual discontinuous Galerkin finite element scheme. 
This is a saddle point system for which Brezzi's Theorem 
can be employed for the analysis of stability.
\begin{lemma}\cite{BS94, wy3655}
 Let ${\cal T}_h$ be a finite element  partition of $\Omega$ satisfying the shape regular assumption given in \cite{BS94, wy3655}. Then, for any $0\leq s \leq 2$ and $1\leq m \leq k$, one 
has
\begin{equation}\label{errorXh}
 \sum_{T\in {\cal T}_h}h_T^{2s}\|u-P_{X_h}u\|^2_{s,T}\leq C h^{2(m+1)}\|u\|^2_{m+1},
\end{equation}
\begin{equation}\label{errorMh}
 \sum_{T\in {\cal T}_h}h_T^{2s}\|u-P_{M_h}u\|^2_{s,T}\leq C h^{2(m-1)}\|u\|^2_{m-1}.
\end{equation}
\end{lemma}

\begin{theorem}
 Assume that the coefficient tensor $a(x)=\{a_{ij}(x)\}_{2 \times 2}$ and $c(x)$ are uniformly piecewise continuous in $\Omega$ with respect to the finite element partition
${\cal T}_h$. Let $u$ and $(u_h;\lambda_h)\in X_h\times M_h$ be the solutions of (\ref{model}) and (\ref{EQ:PD-DG:001}) - (\ref{EQ:PD-DG:002}), respectively. Assume that the exact solution $u$ of
(\ref{model}) is sufficiently regular such that $u\in H^{k+1}(\Omega)$. There exists a constant $C$ such that
\begin{equation*}
 \3bar u_h-P_{X_h} u\3bar+\|\lambda_h-P_{M_h}\lambda\| \leq Ch^{k-1}\|u\|_{k+1},
\end{equation*}
 provided that the meshsize $h<h_0$ holds true for a sufficiently small, but fixed $h_0>0$.
\end{theorem}
\begin{proof}
 It follows from  Lemmas \ref{infsup2} - \ref{coerlemm} that the Brezzi's stability conditions
 are satisfied for the saddle point problem (\ref{err1})-(\ref{err2}). Thus, there exists a constant $C$ such that
\begin{equation}\label{sup3}
 \3bar e_h\3bar+\|\epsilon_h\| \leq C\Bigg(\sup_{v\in{X_h},v\neq 0} \frac{ |-s_h(P_{X_h}u,v)|}{\3bar v\3bar}+
\sup_{p\in M_h,p\neq 0} \frac{|l_u(p)|}{\|p\|}\Bigg).
\end{equation}

Recall that
\begin{equation}\label{sup1}
 \begin{split}
 &\sup_{v\in{X_h},v\neq 0} \frac{ |-s_h(P_{X_h}u,v)|}{\3bar v\3bar} \\
\leq & \sup_{v\in{X_h},v\neq 0} \frac{ | \sum_{e\in \E_h} h_T^{-3} \langle \jump{P_{X_h}u}, \jump{v}
\rangle_e|+| \sum_{e\in \E_h^0} h_T^{-1} \langle \jump{\nabla P_{X_h}u}, \jump{\nabla v}
\rangle_e|}{\3bar v\3bar}\\
 \end{split}
\end{equation}
As to the first term of the right-hand side of (\ref{sup1}), from Cauchy-Schwarz inequality, trace inequality (\ref{tracein}) and (\ref{errorXh}),   we have
\begin{equation}\label{ter2}
 \begin{split}
|\sum_{e\in \E_h} h_T^{-3} &\langle \jump{P_{X_h}u}, \jump{v}
\rangle_e |\leq  C\Big( \sum_{e\in  \E_h}h_T^{-3} \|\jump{P_{X_h}u}\|^2_e\Big)^{\frac{1}{2}}\Big( \sum_{e\in  \E_h}h_T^{-3} \|\jump{v}\|^2_e\Big)^{\frac{1}{2}}
\\ 
 \leq &  C\Big( \sum_{e\in  \E_h}h_T^{-3} (\|\jump{P_{X_h}u}-\jump{ u}\|^2_e+ \|\jump{ u}\|^2_e)\Big)^{\frac{1}{2}}\3bar v\3bar
 \\
\leq & C\Big( \sum_{T\in  {\cal T}_h}h_T^{-4}  \|\jump{P_{X_h}u -  u} \|^2_T +h_T^{-2}  \|\jump{P_{X_h}u -  u}\|^2_{1,T} 
\Big)^{\frac{1}{2}}\3bar v\3bar\\
  \leq & C h^{k-1}\|u\|_{k+1}\3bar v\3bar,
 \end{split}
\end{equation}
where we used $\jump{ u}=0$ as $u\in H^2(\Omega)\cap
H_0^1(\Omega)$.
Similarly, we have
\begin{equation}\label{ter3} 
|\sum_{e\in \E_h^0} h_T^{-1} \langle \jump{\nabla P_{X_h}u}, \jump{\nabla v}
\rangle_e  | \leq   C h^{k-1}\|u\|_{k+1}\3bar v\3bar.
\end{equation} 
Substituting  (\ref{ter2})-(\ref{ter3})  into (\ref{sup1}), we have
\begin{equation}\label{sup4}
\sup_{v\in{X_h},v\neq 0} \frac{| -s_h(P_{X_h}u,v)|}{\3bar v\3bar} \leq   C h^{k-1}\|u\|_{k+1}.
\end{equation}

From Cauchy-Schwarz inequality and (\ref{errorMh}), we obtain
\begin{equation}\label{sup2}
 \begin{split}
&\qquad \sup_{p\in M_h,p\neq 0} \frac{|l_u(p)|}{\|p\|}\\
= &\sup_{p\in M_h,p\neq 0} \frac{\Big|\sum_{T\in{\cal T}_h}\sum_{i,j=1}^2 (a_{ij}(I-P_{M_h})\partial_{ij}^2u,p)_T\Big|}{\|p\|}+  \sup_{p\in M_h,p\neq 0} \frac{\Big| \sum_{T\in{\cal T}_h} (c( I-P_{X_h})u,p)_T\Big|}{\|p\|}\\
 \leq &\sup_{p\in M_h,p\neq 0} \frac{\Big|\|a_{ij}\|_{L^\infty(\Omega)}\Big(\sum_{T\in{\cal T}_h}\sum_{i,j=1}^2 \|(I-P_{M_h})\partial_{ij}^2u\|^2_T\Big)^{\frac{1}{2}}
\Big(\sum_{T\in{\cal T}_h}\|p\|^2_T\Big)^{\frac{1}{2}}\Big|}{\|p\|}\\
&+ \sup_{p\in M_h,p\neq 0} \frac{\Big|\|c\|_{L^\infty(\Omega)}\Big(\sum_{T\in{\cal T}_h} \|(I-P_{X_h}) u\|^2_T\Big)^{\frac{1}{2}}
\Big(\sum_{T\in{\cal T}_h}\|p\|^2_T\Big)^{\frac{1}{2}}\Big|}{\|p\|}\\
&\leq    C h^{k-1}\|u\|_{k+1}+Ch^{k+1}\|u\|_{k+1}\\
&\leq    C h^{k-1}\|u\|_{k+1}.
 \end{split}
\end{equation}

Substituting  (\ref{sup4}) and (\ref{sup2})  into (\ref{sup3}) completes the proof. 
\end{proof}

\section{Bivariate Spline Implementation of Algorithm~\ref
{ALG:primal-dual-dg-fem}}
We notice that $X_h$ is a discontinuous spline space of degree $k$ over 
a finite element partition ${\cal T}_h$ and $M_h$ is a 
discontinuous spline space of degree $k_1$, e.g. $k_1=k-2$ over 
${\cal T}_h$.  When ${\cal T}_h$ is a triangulation, these are spline 
spaces which have  been thoroughly studied in 
\cite{ALW06} and \cite{LS07}.  
In this paper,  let us explain how to use these spline functions 
for numerical solution of the second order elliptic PDE (\ref{model}).  
 When ${\cal T}_h$ is a triangulation, 
spline functions use the Bernstein-B\'ezier representation 
as explained in \cite{LS07}. That is, 
the prime-dual discontinuous Galerkin FEM method discussed in the
previous sections can be reformulated by using 
the Bernstein-B\'ezier representation. 
The representation has several  nice properties (cf. \cite{LS07}): 
(1) the basis functions form a partition of unity, (2) the basis 
functions are nonnegative, 
and (3) the basis functions have explicit formulas for their derivatives, 
integration, their inner product, and triple product integration. 
    
In the remaining of the paper, we use  
both $u\in X_h$ and its coefficient vector ${\bf u}$ in terms 
of Bernstein-B\'ezier representation to write a discontinuous spline 
function $u$. Similarly, we use both  $q\in M_h$  and its  coefficient 
vector  ${\bf q}$.
Most importantly, for any function $u\in X_h$, $u$ is a piecewise 
polynomial function of degree $k$ over
${\cal T}_h$,  the jump function $\jump{u}$ over an interior edge $e$ 
of ${\cal T}_h$  
can be rewritten by using the smoothness conditions between 
the coefficients of two polynomial pieces  $u|_{T_1}$ and $u|_{T_2}$ on 
their common edge $e$ for triangles $T_1, T_2\in {\cal T}_h$ 
which share $e$. See \cite{F86} and \cite{LS07}.  
The smoothness conditions are linear and all  these conditions over each 
interior edge can be expressed together by using $H{\bf u}=0$ as 
explained in \cite{ALW06}, 
where $H$ is a rectangular and sparse matrix and 
${\bf u}$ is the coefficient vector of $u$.  

On the boundary of $\Omega$, $u$ has to satisfy the Dirichlet boundary 
condition which can be approximated by 
using a standard polynomial interpolation method, i.e.,  
$u(\bfx)|_e=g(\bfx)$ for $k+1$ distinct points 
$\bfx \in e$, where $e$ is a boundary edge of ${\cal T}_h$. As $u$ is a 
polynomial on $e$, the interpolation 
condition $u(\bfx)|_e=g(\bfx)$ can be expressed by linear equations in 
terms of its coefficients. We put these
linear equations for all boundary edges together and express them by 
$B{\bf u}= {\bf g}$, where $B$ is a 
rectangular and sparse matrix and  ${\bf g}$ is a vector consisting of
the  values of $g$ at the $k+1$ equally-spaced points over  $e$ for all 
boundary edges $e\in \triangle$. 

The PDE equation in (\ref{weak}) can be discretized 
by using Bernstein-B\'ezier
representation as follows. We first approximate the right-hand side 
$f$ by discontinuous spline functions in $S_f\in M_h$. For example, we 
may choose $S_f$ to be the piecewise polynomial 
function which interpolates $f$ at the domain points on $T$ of degree 
$k_1$ for all triangle $T\in {\cal T}_h$, under the assumption that $f$ 
is a continuous function. For another example, we 
choose $S_f\in M_h$ such that for each triangle $T\in {\cal T}_h$, 
\begin{equation}
\label{testspace2}
\int_T f q dxdy = \int_T S_f q dxdy, \quad \forall q\in {\cal P}_{k_1},
\end{equation}
where ${\cal P}_{k_1}$ is the standard polynomial space of total degree
$k_1$. It is easy to know that the problem (\ref{testspace2}) has a 
unique solution
of $S_f|_T$. Thus, $S_f\in {\cal M}_h$ is well-defined. 
In fact, we have the following properties
\begin{equation}
\label{property1}
\|S_f\|\le \|f\| \hbox{ and } \|S_f- f\|= \min_{s\in M_h}\|s- f\|.
\end{equation}
Indeed, we have $\int_T |S_f|^2dxdy= \int_T fS_f dxdy$ for all 
$T\in {\cal T}_h$ and 
use Cauchy-Schwarz inequality to have the inequality in 
(\ref{property1}). The equality in (\ref{property1}) can be seen from 
the solution of the least squares problem in (\ref{testspace2}).

We compute the inner product integration on the right-hand of 
(\ref{weak}) exactly by using Theorem 2.34 
in \cite{LS07} and a triple inner product formula. That is, we have
$$
\int_\Omega f qdxdy = \int_\Omega S_f qdxdy =\langle M{\bf f}, 
{\bf q}\rangle, 
$$
where ${\bf f}$ is the coefficient vector of $S_f$,  $M$ is called 
the mass matrix which is a blockly diagonal matrix and ${\bf q}$ is the 
coefficient vector of $q$. 
 
Similarly, we approximate the coefficients $a_{ij}$ 
by discontinuous spline functions in another discontinuous spline 
space $S_{ij}\in L_h=S^{-1}_{1}({\cal T}_h)$ of degree $1$, 
say piecewise linear interpolation of $a_{ij}$. 
\begin{equation}
\label{testspace1}
 \int_T a_{ij} \partial^2_{ij} u q dxdy \approx \int_T S_{i,j} 
 \partial^2_{ij} u q dxdy, \quad \forall u\in {\cal P}_{k}, q\in 
{\cal P}_{k-2}.
\end{equation}  
 Once we have $S_{ij}$, we compute triple product 
integration on the left-hand side of (\ref{weak}). That is, 
$\int_T S_{ij}\partial_{ij}^2 u qdxdy$
has an exact formula in terms of the coefficients of $S_{ij}, u,$ 
and  $q$. 
Thus we have
$$
\int_\Omega \sum_{i,j=1}^2 a_{ij} \partial^2_{ij} u q dxdy \approx
\int_\Omega \sum_{i,j=1}^2 S_{ij} 
\partial^2_{ij} u q dxdy =\langle K {\bf u}, {\bf q}\rangle,  
$$
where $K$ is the stiffness matrix related to the PDE (\ref{model}). 

In order to have an equality in the above formula, we now use the 
standard $L^2$ projection $P_{M_h}$ which is defined by $P_{M_h}(v)
\in M_h$ such that 
\begin{equation}
\label{L2projection}
\langle P_{M_h}(v), q\rangle = \langle v, q\rangle, \forall q\in M_h.
\end{equation}
Thus, we have 
$$
\int_\Omega \sum_{i,j=1}^2 a_{ij} \partial^2_{ij} u q dxdy =\langle P(\sum_{i,j=1}^2 
a_{ij} \partial^2_{ij} u), q\rangle 
=\int_\Omega \sum_{i,j=1}^2 P_{M_h}(a_{ij} \partial^2_{ij}u) qdxdy. 
$$
Since the the projection is linear, we can write 
$$
\int_\Omega \sum_{i,j=1}^2 
P_{M_h}(a_{ij} \partial^2_{ij}u) qdxdy =\langle K{\bf u}, {\bf 
q}\rangle, 
$$
for a blockly diagonal matrix $K$ and for all $q\in M_h$. 
In this way,  we obtain a  discretized PDE equation: 
$\langle K {\bf u}, {\bf q}\rangle =  \langle M {\bf f},{\bf q}\rangle$ 
for all ${\bf q}\in \mathbb{R}^{d({\cal M}_h)}$ or a linear system:
\begin{equation}
\label{discretePDE}
 K {\bf u} =  M{\bf f}.
\end{equation}
Note that both $M$ and $K$ can be computed in parallel.

In terms of the Berstein-B\'ezier representation, the bilinear forms 
in (\ref{bilinearform}) and 
(\ref{bilinearform2}) can be rewritten as 
\begin{equation}
\label{newbilinearform}
s(u, v) =h^2\langle H {\bf u}, H{\bf v}\rangle + h^2\langle B {\bf u},
B{\bf v}\rangle, \quad \forall u, v\in X_h,
\end{equation}
and 
\begin{equation}
\label{newbilinearform2}
b(u, q)= \langle K{\bf u}, {\bf q}\rangle, \quad \forall u\in X_h, q\in M_h.
\end{equation}
With the above preparation, Algorithm~\ref{ALG:primal-dual-dg-fem} 
can be recast as follows. 

Let us consider the following minimization problem for (\ref{weak}): 
Find ${\bf u}$ satisfying
\begin{equation}
\label{minH-ww} \min  \frac{h^2}{2}(\|H{\bf
u}\|^2 + \|B{\bf u}-{\bf g}\|^2),  \quad \hbox{subject to }
K{\bf u}=M {\bf f}.
\end{equation}
Note that the boundary condition is imposed by minimizing
the error in an least-squares sense so that the boundary conditions do 
not need to be strictly enforced.

This minimization problem (\ref{minH-ww}) can be reformulated by using 
Lagrange multiplier method as follows: let 
\begin{equation}
\label{Lagfun-ww} L({\bf u}, \lambda)
= \frac{h^2}{2}(\|H {\bf u}\|^2+
\|B{\bf u}-{\bf g}\|^2) + {\bf \lambda}^\top (K{\bf u}- M{\bf f}),
\end{equation}
where ${\bf \lambda}$ is a Lagrange multiplier. Thus,  the minimizer 
${\bf u}^*$ of (\ref{minH-ww}) satisfies (\ref{minH2-ww}). Hence, we have
\begin{algorithm}\emph{(The Primal-Dual Bivariate Spline Method)} 
\label{newalg}
Find a vector pair $({\bf u}^*, {\bf \lambda}^*)
\in  \mathbb{R}^{d(X_h)}\times \mathbb{R}^{d(M_h)}$ satisfying 
\begin{equation}
\label{minH2-ww}
\begin{cases}
h^2\langle H{\bf u}^*, H{\bf d}\rangle + h^2\langle B{\bf
u}, B{\bf d}\rangle+ \langle \lambda^*, K {\bf d}\rangle &= h^2\langle
{\bf g}, B{\bf d}\rangle, \qquad \forall {\bf d} \in 
\mathbb{R}^{d(X_h)}, \cr
\langle {\bf q} , K {\bf u}^*\rangle &= \langle {\bf q}, 
M {\bf f}\rangle, \qquad \forall {\bf q} \in \mathbb{R}^{d(M_h)}, 
\end{cases}
\end{equation} 
where $d(X_h)$ is the dimension of $X_h$ and $d(M_h)$ is the dimension 
of $M_h$. In fact, 
$d(X_h)= (k+1)(k+2) N({\cal T}_h)/2$ 
and $d(M_h)= (k_1+1)(k_1+2) N({\cal T}_h)/2$ 
with $N({\cal T}_h)$ being the number of triangles in ${\cal T}_h$.
We shall denote by $u_h\in X_h$ the spline solution with coefficient 
vector ${\bf u}^*$ and similarly, 
$\lambda_h\in M_h$ with coefficient vector $\lambda^*$.  
\end{algorithm}

This Algorithm~\ref{newalg} will be implemented and  numerically 
experimented in this paper. 

\section{Numerical Results based on Minimization (\ref{minH-ww})
\label{sec:nr}}
We have implemented  Algorithm~\ref{newalg} 
in MATLAB based on the spline function 
implementation method discussed in \cite{ALW06} which is completely different from the spline functions
implemented in \cite{S15}. Our main driving code is given below.
\begin{verbatim}
[V,T]=mytriangulation(3); %input a triangulation.
k=5; % degree of spline functions in X_h, You can use any integer d\ge 2.
k1=d-2; %degree of spline functions in M_h.
caseNum=101; %this number is the number in the test function list g2.  
%The following three lines generate 1 million points for computing the RMSE.
xmin =min(V(:,1)); xmax=max(V(:,1)); xscal=xmax-xmin; 
ymin =min(V(:,2)); ymax=max(V(:,2)); yscal=ymax-ymin;
[X,Y]=meshgrid(xmin:xscal/1000:xmax,ymin:yscal/1000:ymax);
%The exact solution: values and its derivatives.
[Exact,Exactdx,Exactdy,Exactxx,Exactxy,Exactyy]=g2(X(:),Y(:),caseNum);
L=5; %Level of refinement.  
k2=2; %degree of splines to approximate the PDE coefficients. 
for i=1:L
 [V,T]=refine(V,T); %uniform refinement of triangulation (V, T).
 [E,TE,TV,EV,B] = tdata(V,T); %triangulation relations
 H0 = smoothness(V,T,k,0); %The function value continuity conditions 
 H1 = smoothness1(V,T,k,1);%The function derivative continuity conditions
 H=[H0;H1]; %or simply use:   H=smoothness(V,T,d,1);
 M=mass(V,T,k1); %the mass matrix
[A,B,C,D,E] = bnetw2(V,T,k2,'weights4PDE2',caseNum); 
%spline approximation of PDE coefficients above and  the weighted stiff matrices
[Kxx,Kxy,Kyx,Kyy,Kc]=bending2(V,T,k,k1,k2,A,B,C,D,E); 
 K=Kxx+Kxy+Kyx+Kyy; %use K for simplicity. 
 F = bnet2(V,T,k1,'ellipticPDE',caseNum); %spline for the right-hand side of PDE
 [G,Bm] = dirchlet(V,T,TE,E,k,'g2',caseNum); %generating the boundary conditions
%The following 3 lines are parameters for constrained iterative minimization. 
 m = (k+1)*(k+2)/2; n = size(T,1);
 p1 = length(G); p2 = size(H,2); 
 tol=1.0e-16;    eps=6;    max_it=5;
%The following CImin is an algorithm explained in Appendix 1. 
    c=CImin(H'*H+Bm'*Bm,K'-Kc',Bm'*G,M*F,eps,max_it, tol); % 
% We now evaluate the spline approximation of the solution:
 s{1}=V;s{2}=T;s{3}=c; %s is a spline of coefficients c over triangulation (V,T).
%computing the derivatives of s.
 cx=xder(s); cy=yder(s);cxx=xder(cx); cxy=yder(cx); cyy=yder(cy); 
 Z = gevalP(s,X(:),Y(:)); %evaluation for checking accuracy and displaying
 Zx = gevalP(cx,X(:),Y(:)); Zy = gevalP(cy,X(:),Y(:));
 Zxx=gevalP(cxx,X(:),Y(:));Zxy=gevalP(cxy,X(:),Y(:));Zyy=gevalP(cyy,X(:),Y(:));   
% Finally we check the accuracy in the root mean square errors. 
 err=sqrt(mse(Exact(:)-Z));
 errx=sqrt(mse(Exactdx(:)-Zx)); erry=sqrt(mse(Exactdy(:)-Zy));
 errxx=sqrt(mse(Exactxx(:)-Zxx)); errxy=sqrt(mse(Exactxy(:)-Zxy));
 erryy=sqrt(mse(Exactyy(:)-Zyy);
end
\end{verbatim}


We let $S_u$ be the spline solution with the
coefficient vector ${\bf c}(u)$ which is the minimizer of
(\ref{minH-ww}) and report the root mean squared error (RMSE) of
$u-S_u$, $\nabla (u - S_u) =  (\dfrac{\partial}{\partial x}(u-S_u),
\dfrac{\partial}{\partial y}(u-S_u))$ and $\nabla^2(u-S_u)=
(\dfrac{\partial^2}{\partial x^2}(u-S_u),
\dfrac{\partial^2}{\partial x\partial y}(u-S_u),
\dfrac{\partial^2}{\partial y^2}(u-S_u))$ based on their values 
over $1001\times 1001$
equally-spaced points over $\Omega$. More precisely, we report the
RMSE of $\nabla (u - S_u)$ which is the average of the RMSE of
$\dfrac{\partial}{\partial x}(u-S_u)$ and $\dfrac{\partial}{\partial
y}(u-S_u)$. Similar for the RMSE of $\nabla^2(u-S_u)$. We shall also
present the rates of convergence of RMSE between refinement levels. 

The remaining of this section is divided into three subsections. 
In the first subsection, we present  numerical results based on 
the PDE with smooth coefficients and $c\equiv 0$. 
We also use smooth solutions to test our spline method. 
One of purposes is 
to demonstrate that our MATLAB implementation is correct and 
is able to produce excellent numerical solution. Another purpose 
is to compare with the numerical results in \cite{SS13} and \cite{WW15}.

In the next two subsections, 
we mainly present numerical results from the 
second order elliptic PDE with nonsmooth coefficients and nonsmooth  
solution.  
Our numerical experiments show that 
the higher order splines still give a better approximation than the 
lower order splines. Our numerical results for the same testing 
function used in \cite{SS13} provide more evidence than the convergence
rate for $|u- S_u|_{H^2(\Omega)}$ is $d-1$ for $d=6, 7, 8$. 

Finally we  show some spline solutions for PDE in (\ref{model}) 
with nonzero function $c$ for smooth and nonsmooth exact solutions. 
Numerical results are similar to the case when $c\equiv 0$.

\subsection{The case with smooth coefficients}
\label{ss:smooth}
In the following examples, we shall use spline spaces
$S^{-1}_d(\triangle_\ell)$ of various degrees $d=2, 3, 4, 5, 6, 7,
8...$ to solve the PDE of interest, where $\triangle_0$ is a standard 
triangulation of $\Omega$
and $\triangle_\ell$ is the uniform refinement of
$\triangle_{\ell-1}$ for $\ell=1, 2, 3, 4$. 

\begin{example}
\label{ex0} We begin with a 2nd order elliptic equation with
constant coefficients and smooth solution $u =\sin(x)\sin(y)$ 
which satisfies the following partial differential equation:
\begin{equation}
\label{PDEex1} 3\frac{\partial^2}{\partial x^2} u +
2\frac{\partial^2}{\partial x\partial y} u
+2\frac{\partial^2}{\partial y^2} u = f(x,y), \quad (x,y)\in
\Omega\subset \mathbb{R}^2,
\end{equation}
where $\Omega$ is a standard square domain $[0,1]^2$  (cf.
\cite{WW15}).  We use $X_h=S^{-1}_d(\triangle_\ell)$
and $M_h=S^{-1}_{d-2}(\triangle_\ell)$ with $h=|\triangle_\ell|$.  
We use a triangulation $\triangle_0$ which consists of
2 triangles and then uniformly refine $\triangle_0$ repeatedly 
to obtain $\triangle_\ell, \ell=1, 2, 3, 4, 5$.

\begin{table}[thbp]
\centering
\begin{tabular}{|c|c|r|c|r|c|r|}
\hline $|\triangle|$ & $u-S_u$ & rate & $\nabla(u-S_u)$ & rate &
$\nabla^2(u- S_u)$ & rate \cr\hline 
0.7071 & 2.052453e-03 & 0.00 & 1.564506e-02 & 0.00 & 1.163198e-01 & 0.00  \cr \hline 
0.3536 & 7.574788e-04 & 1.44 & 4.728042e-03 & 1.72 & 6.078911e-02 & 0.94  \cr \hline 
0.1768 & 2.779251e-04 & 1.45 & 1.397469e-03 & 1.76 & 3.022752e-02 & 1.01  \cr \hline 
0.0884 & 8.156301e-05 & 1.77 & 3.809472e-04 & 1.88 & 1.489634e-02 & 1.03  \cr \hline 
0.0442 & 2.161249e-05 & 1.92 & 9.836874e-05 & 1.95 & 7.401834e-03 & 1.01  \cr \hline  
\end{tabular}
\caption{The RMSE of spline solutions using $X_h=S^{-1}_2(\triangle_\ell)$ and 
$M_h=S^{-1}_0(\triangle_\ell)$ for $
\ell=1, 2, 3, 4, 5$ of PDE (\ref{PDEex1})\label{T0d2}}
\end{table}

Table~\ref{T0d2} may be compared with  Table 8.1 in \cite{WW15}. 
First of all, we can see that there is a superconvergence in 
$L^2$ norm approximation  in Table 8.1 in \cite{WW15}. 
That is, the convergence rate in \cite{WW15} 
is about 4 although they only use piecewise polynomials of degree 2. 
So far there is no mathematical theory to guarantee this 
superconvergence.  Note that the computation of their convergence is based on node points
of the underlying triangulation, that is, 6 points per triangle for all triangles in ${\cal T}_h$ for each $h>0$.
In our Table~\ref{T0d2}, the convergence is measured in 
the RMSE based on $1001\times 1001$ equally-spaced points over $\Omega$ 
and our convergence rate is about 2 for $M_h=S^{-1}_0(\triangle_\ell)$. 
Nevertheless, our convergence of $\nabla (u_h- u)$ is better 
than that in Table 8.1 in \cite{WW15}. Also, we are able to show 
the convergence in the second order derivatives of $u- u_h$, i.e. the 
semi-norm $|u-u_h|_{H^2(\Omega)}$.   

In the next few tables, we use $X_h=S^{-1}_k(\triangle_\ell)$ and 
$M_h=S^{-1}_{k_1}(\triangle_\ell)$ with $k_1\ge 1$. Then the order of convergence will increase.   
This is an advantage of our numerical algorithm  that we can use polynomials of
higher degree easily by simply adjusting $k$ and/or $k_1$ 
in our main driving code. For $k=3$ and $k_1=1$, we have 

\begin{table}[htbp]
\centering
\begin{tabular}{|c|c|r|c|r|c|r|}
\hline $|\triangle|$ & $u-S_u$ & rate & $\nabla(u-S_u)$ & rate &
$\nabla^2(u- S_u)$ & rate \cr\hline 
0.7071 & 1.549234e-03 & 0.00 & 5.551342e-03 & 0.00 & 2.571257e-02 & 0.00  \cr \hline 
0.3536 & 3.614335e-04 & 2.10 & 1.266889e-03 & 2.13 & 6.506533e-03 & 1.99  \cr \hline 
0.1768 & 8.995656e-05 & 2.01 & 3.098134e-04 & 2.03 & 1.627964e-03 & 2.00  \cr \hline 
0.0884 & 2.255287e-05 & 2.00 & 7.741892e-05 & 2.00 & 4.087224e-04 & 1.99  \cr \hline 
0.0442 & 5.639105e-06 & 2.00 & 1.935553e-05 & 2.00 & 1.026039e-04 & 1.99  \cr \hline 
\end{tabular}
\caption{The RMSE of spline solutions using 
$X_h=S^{-1}_3(\triangle_\ell)$ and $M_h=S^{-1}_1(\triangle_\ell)$ for $
\ell=1, 2, 3, 4, 5$ of PDE (\ref{PDEex1})\label{T0d3}}
\end{table}

To increase the convergence rates
for $u-u_h$ and $\nabla (u-u_h)$, we use $k_1=k$ which can be easily
adjusted in our main driving code. As we can see from Table~\ref{T0d33}.
The convergence and convergence rates are much better than 
Tables~\ref{T0d2} and \ref{T0d3}. 

\begin{table}[htbp]
\centering
\begin{tabular}{|c|c|r|c|r|c|r|}
\hline $|\triangle|$ & $u-S_u$ & rate & $\nabla(u-S_u)$ & rate &
$\nabla^2(u- S_u)$ & rate \cr\hline 
0.7071 & 1.544907e-04 & 0.00 & 1.004675e-03 & 0.00 & 9.443382e-03 & 0.00  \cr \hline 
0.3536 & 1.044383e-05 & 3.89 & 1.351050e-04 & 2.89 & 2.474539e-03 & 1.94  \cr \hline 
0.1768 & 8.189057e-07 & 3.67 & 1.757983e-05 & 2.94 & 6.360542e-04 & 1.97  \cr \hline 
0.0884 & 8.172475e-08 & 3.32 & 2.226705e-06 & 2.98 & 1.612220e-04 & 1.98  \cr \hline 
0.0442 & 8.968295e-09 & 3.19 & 2.803368e-07 & 2.99 & 4.053880e-05 & 1.99  \cr \hline  
\end{tabular}
\caption{The RMSE of spline solutions using 
$X_h=S^{-1}_3(\triangle_\ell)$ and 
$M_h=S^{-1}_3(\triangle_\ell)$ for $
\ell=1, 2, 3, 4, 5$ of PDE (\ref{PDEex1})\label{T0d33}}
\end{table}

Similarly, we can use $k=4$ and $k_1=4$. The numerical results are given 
in Tables~\ref{T0d44}--\ref{T0d55} and show that the convergence rate is more than $k=4$. 

\begin{table}[thbp]
\centering
\begin{tabular}{|c|c|r|c|r|c|r|}
\hline $|\triangle|$ & $u-S_u$ & rate & $\nabla(u-S_u)$ & rate &
$\nabla^2(u- S_u)$ & rate \cr\hline 
0.7071 & 7.146215e-06 & 0.00 & 8.190007e-05 & 0.00 & 1.185424e-03 & 0.00  \cr \hline 
0.3536 & 2.645725e-07 & 4.76 & 5.224157e-06 & 3.97 & 1.449168e-04 & 3.03  \cr \hline 
0.1768 & 1.316127e-08 & 4.33 & 3.160371e-07 & 4.05 & 1.685747e-05 & 3.10  \cr \hline 
0.0884 & 6.399775e-10 & 4.36 & 1.937981e-08 & 4.03 & 1.987492e-06 & 3.08  \cr \hline 
0.0442 & 2.456211e-11 & 4.70 & 1.200460e-09 & 4.01 & 2.409873e-07 & 3.04  \cr \hline 
\end{tabular}
\caption{The RMSE of spline solutions using 
$X_h=S^{-1}_4(\triangle_\ell)$ and $M_h=S^{-1}_4(\triangle_\ell)$ for $
\ell=1, 2, 3, 4, 5$ of PDE (\ref{PDEex1})\label{T0d44}}
\end{table}

\begin{table}[thbp]
\centering
\begin{tabular}{|c|c|r|c|r|c|r|}
\hline $|\triangle|$ & $u-S_u$ & rate & $\nabla(u-S_u)$ & rate &
$\nabla^2(u- S_u)$ & rate \cr\hline 
0.7071 & 2.760695e-07 & 0.00 & 3.427271e-06 & 0.00 & 5.952484e-05 & 0.00  \cr \hline 
0.3536 & 4.721134e-09 & 5.87 & 1.113495e-07 & 4.94 & 3.938359e-06 & 3.92  \cr \hline 
0.1768 & 7.777767e-11 & 5.92 & 3.351050e-09 & 5.05 & 2.373035e-07 & 4.05  \cr \hline 
0.0884 & 2.394043e-12 & 5.02 & 1.026261e-10 & 5.03 & 1.447321e-08 & 4.04  \cr \hline 
\end{tabular}
\caption{The RMSE of spline solutions using 
$X_h=S^{-1}_5(\triangle_\ell)$ and 
$M_h=S^{-1}_5(\triangle_\ell)$ for $
\ell=1, 2, 3, 4, 5$ of PDE (\ref{PDEex1})\label{T0d55}}
\end{table}

Note that in the last row of  Table~\ref{T0d55}, 
the rate of convergence in $L_2$ norm is 5.02 which is  lower than 
 5.92. This is because the iterative solution of the linear system 
achieves the machine precision for this test function using MATLAB.  
Indeed, if we use $u=\sin(2\pi x)\sin(2\pi y)$ which is slightly 
harder to approximate than $u=\sin(x)\sin(y)$ , the 
rate of convergence will be around 6. See the rates of convergence in  
the RMSE of the spline solution shown in Table~\ref{T0d55a}, 
where the rate is 5.74.   

\begin{table}[thbp]
\centering
\begin{tabular}{|c|c|r|c|r|c|r|}
\hline $|\triangle|$ & $u-S_u$ & rate & $\nabla(u-S_u)$ & rate &
$\nabla^2(u- S_u)$ & rate \cr\hline 
0.7071 & 2.390050e-02 & 0.00 & 2.699640e-01 & 0.00 & 4.628174e+00 & 0.00  \cr \hline 
0.3536 & 4.997698e-04 & 5.58 & 1.076435e-02 & 4.66 & 3.787099e-01 & 3.63  \cr \hline 
0.1768 & 8.812568e-06 & 5.83 & 3.225226e-04 & 5.06 & 2.356171e-02 & 3.99  \cr \hline 
0.0884 & 1.648941e-07 & 5.74 & 8.620885e-06 & 5.22 & 1.260638e-03 & 4.20  \cr \hline 
\end{tabular}
\caption{The RMSE of spline solutions using $X_h=M_h=S^{-1}_5(\triangle_\ell)$  
for $\ell=1, 2, 3, 4$ of PDE (\ref{PDEex1}) with 
$u=\sin(2\pi x)\sin(2\pi y)$. \label{T0d55a}}
\end{table}

\begin{table}[thbp]
\centering
\begin{tabular}{|c|c|r|c|r|c|r|}
\hline $|\triangle|$ & $u-S_u$ & rate & $\nabla(u-S_u)$ & rate &
$\nabla^2(u- S_u)$ & rate \cr\hline 
0.7071 & 1.862502e-03 & 0.00 & 2.436280e-02 & 0.00 & 4.404080e-01 & 0.00  \cr \hline 
0.3536 & 5.460275e-05 & 5.09 & 1.350238e-03 & 4.18 & 5.202210e-02 & 3.08  \cr \hline 
0.1768 & 5.354973e-07 & 6.67 & 2.432368e-05 & 5.79 & 1.842914e-03 & 4.80  \cr \hline 
0.0884 & 3.836807e-09 & 7.12 & 4.105804e-07 & 5.89 & 6.417771e-05 & 4.84  \cr \hline 
\end{tabular}
\caption{The RMSE of spline solutions using $X_h=M_h=S^{-1}_6(\triangle_\ell)$ 
for $\ell=1, 2, 3, 4$ of PDE (\ref{PDEex1}) with 
$u=\sin(2\pi x)\sin(2\pi y)$. \label{T0d66a}}
\end{table}

\begin{table}[thbp]
\centering
\begin{tabular}{|c|c|r|c|r|c|r|}
\hline $|\triangle|$ & $u-S_u$ & rate & $\nabla(u-S_u)$ & rate &
$\nabla^2(u- S_u)$ & rate \cr\hline 
0.7071 & 1.167121e-03 & 0.00 & 2.022185e-02 & 0.00 & 5.174476e-01 & 0.00  \cr \hline 
0.3536 & 4.520586e-06 & 8.01 & 1.575837e-04 & 7.01 & 8.136845e-03 & 6.00  \cr \hline 
0.1768 & 2.063180e-08 & 7.78 & 1.352130e-06 & 6.87 & 1.347347e-04 & 5.92  \cr \hline 
0.0884 & 9.814292e-11 & 7.72 & 1.032652e-08 & 7.03 & 1.947362e-06 & 6.10  \cr \hline 
\end{tabular}
\caption{The RMSE of spline solutions using $X_h=M_h=S^{-1}_7(\triangle_\ell)$ 
for $\ell=1, 2, 3, 4$ of PDE (\ref{PDEex1}) with 
$u=\sin(2\pi x)\sin(2\pi y)$. \label{T0d77a}}
\end{table}

We have tested other solutions (e.g. $u=1/(1+x^2+y^2), 
u=\sin(\pi x)\sin(\pi y), u =\sin(\pi (x^2+y^2))$ 
and etc.. Numerical results are similar to 
Tables~\ref{T0d55a}--Tables~\ref{T0d77a}.  
We can see that the rate of convergence in 
$L_2$ norm is optimal for $d\ge 5$ and for sufficiently smooth 
solutions.  That is,  the optimal convergence rate is reached  
when using splines in $S^1_d(\triangle)$ with $d\ge 5$. 

Finally, our algorithm is efficient in the following sense: each table 
above (Tables~\ref{T0d55}-- \ref{T0d77a} is generated within 
$180$ seconds based on a desktop computer of 16GB in RAM with Intel 
Processor i7-3770CPU @3.4GHz speed. 
For Tables~\ref{T0d2}--\ref{T0d44}, it takes 550 seconds to generate. 
Major time is spent on the evaluation of $1001\times 1001$ 
spline values. 
\end{example}

\subsection{The case with nonsmooth coefficients and nonsmooth solution}
\label{ss:nonsmooth}
The numerical results in the previous subsection show that our program 
works well to find numerical solution of the PDE with smooth 
coefficients. 
In this subsection, we shall demonstrate that our method works
well for those PDE with nonsmooth coefficients which can not be 
converted into its divergence form.
Higher order of splines produce more accurate solutions 
in $L_2$ norm and $H^1$ semi-norm. Even the solution is only 
$C^1(\Omega)$, we are  able to approximate
the solution in $H^2(\Omega)$ semi-norm very well as the same as
in \cite{SS13}.

\begin{example}
\label{ex3} In this example, we show the performance of our spline
solutions for a PDE with  nondifferentiable coefficients and
nonsmooth exact solution
$u=xy(e^{1-|x|}-1)(e^{1-|y|}-1)$ which satisfies
\begin{equation}
\label{PDEex3} 2\frac{\partial^2}{\partial x^2} u +
2\sign(x)\sign(y)\frac{\partial^2}{\partial x\partial y} u
+2\frac{\partial^2}{\partial y^2} u = f(x,y), \quad (x,y)\in
\Omega\subset \mathbb{R}^2
\end{equation}
where $u=0$ on the boundary of  $\Omega=[-1, 1]\times [-1, 1]$ as in
\cite{SS13}. Note that the solution is in $H^2(\Omega)$, but not
continuously twice differentiable. We shall use $X_h= S^{-1}_d(\triangle_\ell)$ and 
$M_h=S^{-1}_{d-2}(\triangle_\ell)$ with $\triangle_\ell$ 
shown in Figure~\ref{tris}.

\begin{figure}[htbp]
\begin{center}
\includegraphics[width=7cm, height=6cm, angle=0]{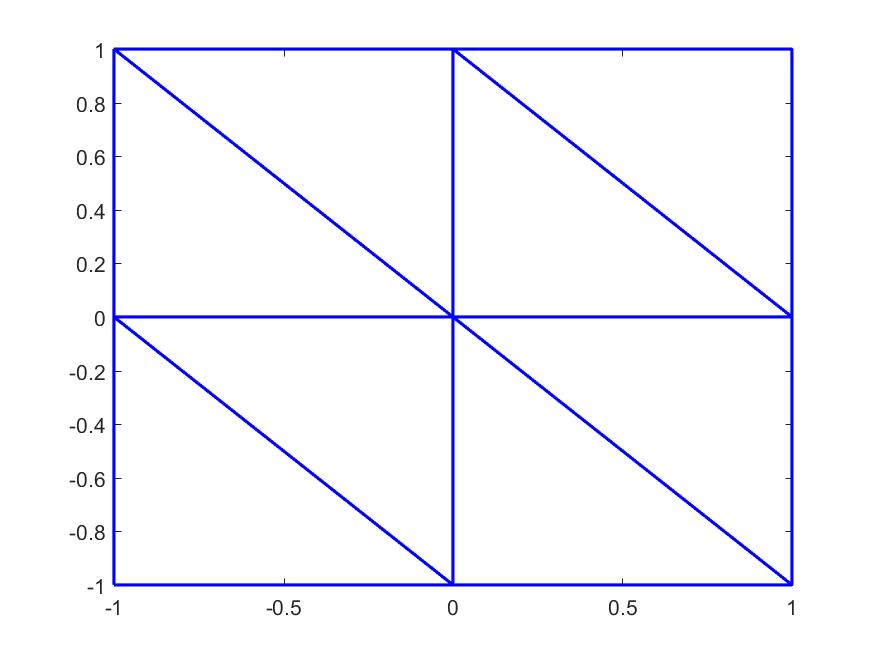}
\includegraphics[width=7cm, height=6cm, angle=0]{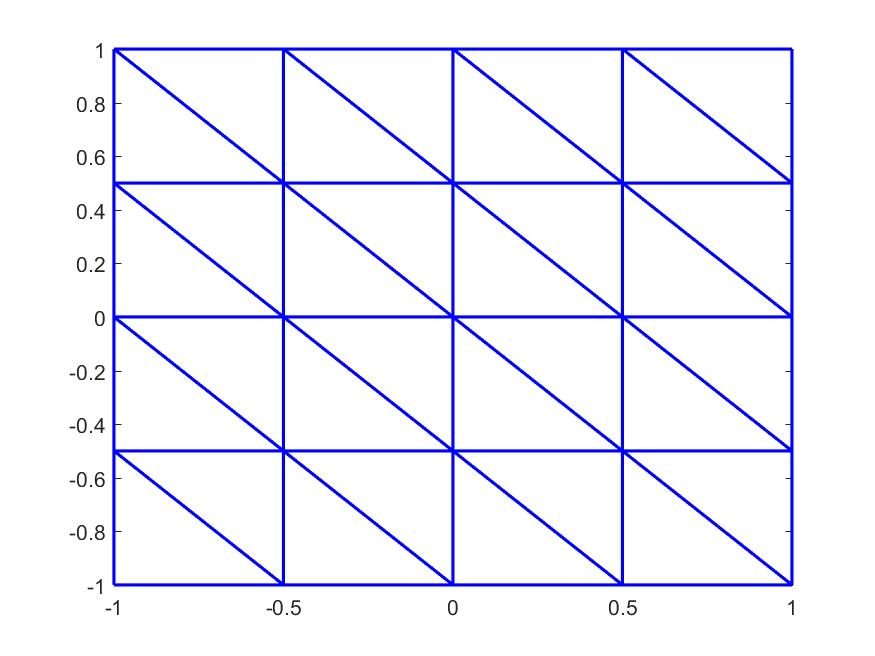}
\includegraphics[width=7cm, height=6cm, angle=0]{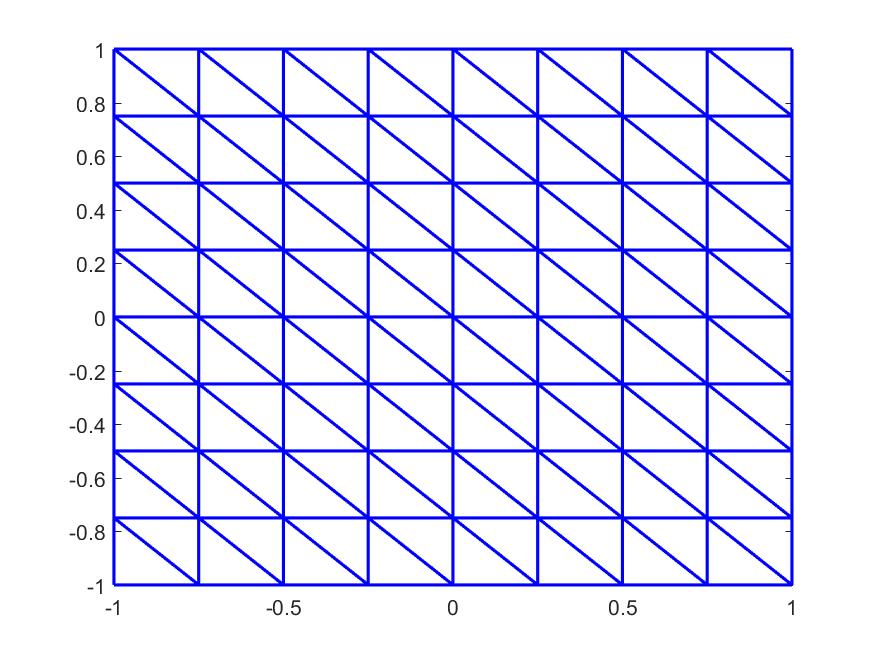}
\includegraphics[width=7cm, height=6cm, angle=0]{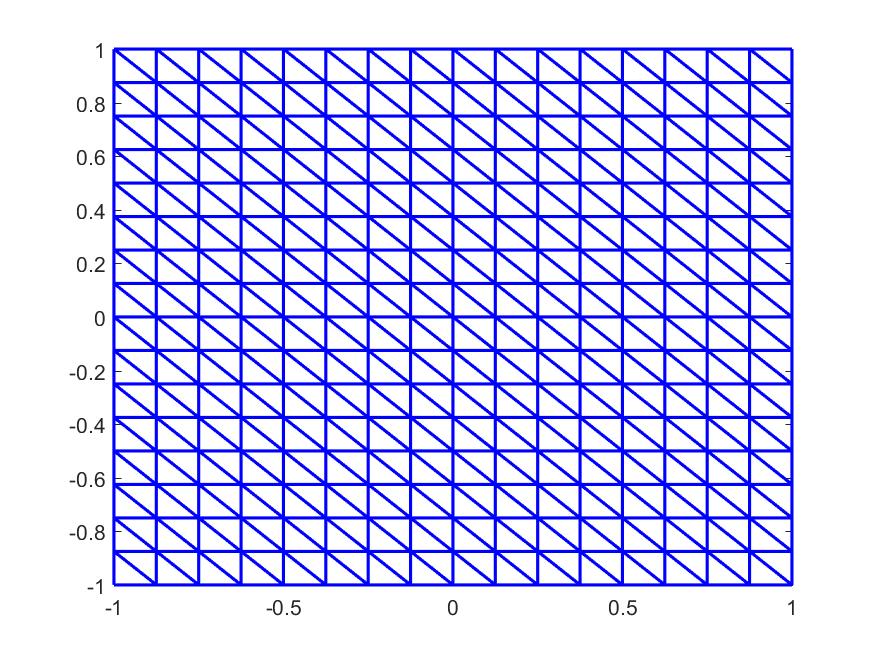}
\caption{Triangulations $\triangle_\ell, \ell=0, 1, 2, 3$\label{tris}}
\end{center}
\end{figure}

We can see that the convergences of our spline method in 
Tables~\ref{ndexd2}
are better than those in Table 8.5 in \cite{WW15} and the gradient 
approximations are better than the corresponding gradients in Table 8.6 
in \cite{WW15}. However, our $L_2$ estimates are not able to 
achieve the accuracy in Table 8.6 in \cite{WW15} using quadratic 
splines. More accurate solutions are obtained when splines of higher 
degrees are used. See Tables~\ref{ndexd3}--\ref{ndexd6}. 
This example was initially studied in \cite{SS13}.
See Fig. 2 in \cite{SS13}. For $d=2$, we are able to achieve the 
accuracy of $6.841801e-02$ at the size $h=0.0313$ for the accuracy 
of $\nabla^2(u-S_u).$ 
 
Instead of showing the convergence rates of 
$|u- u_h|_{H^2(\Omega)}$ in a semi-log graph for various $d=2, 3, 4, 5$ 
in \cite{SS13}, we present more detailed numerical convergence 
in root mean squared error (RMSE) for $u- u_h, \nabla (u- u_h), $ 
as well as $|D_x^2(u- u_h)|+|D_xD_y(u-u_h)|+|D^2_y(u-u_h)$ 
based on $333\times 333$
equally-spaced points over $\Omega=[-1,1]\times [-1, 1]$.

\begin{table}[thbp]
\centering
\begin{tabular}{|c|c|r|c|r|c|r|}
\hline $|\triangle|$ & $u-S_u$ & rate & $\nabla(u-S_u)$ & rate &
$\nabla^2(u- S_u)$ & rate \cr\hline
0.7071 & 4.996969e-02 & 0.00 & 1.180311e-01 & 0.00 & 4.068785e-01 & 0.00  \cr \hline 
0.3536 & 4.359043e-02 & 0.20 & 1.014477e-01 & 0.22 & 3.228522e-01 & 0.33  \cr \hline 
0.1768 & 2.742714e-02 & 0.67 & 6.679062e-02 & 0.60 & 2.187010e-01 & 0.56  \cr \hline 
0.0884 & 1.192668e-02 & 1.20 & 3.080064e-02 & 1.12 & 1.122648e-01 & 0.96  \cr \hline 
\end{tabular}
\caption{The RMSE of spline solutions using the pair
$X_h=S^{-1}_2(\triangle_\ell), M_h=S^{-1}_0(\triangle_\ell)$ of spline spaces for $
\ell=0, 1, 2, 3, 4$ of PDE (\ref{PDEex3}) based on uniform
triangulations in Figure~\ref{tris}\label{ndexd2}}
\end{table}

\begin{table}[thbp]
\centering
\begin{tabular}{|c|c|r|c|r|c|r|}
\hline $|\triangle|$ & $u-S_u$ & rate & $\nabla(u-S_u)$ & rate &
$\nabla^2(u- S_u)$ & rate \cr\hline
0.7071 & 2.251481e-02 & 0.00 & 5.717663e-02 & 0.00 & 2.307563e-01 & 0.00  \cr \hline 
0.3536 & 3.531272e-03 & 2.67 & 1.098214e-02 & 2.38 & 5.867175e-02 & 1.98  \cr \hline 
0.1768 & 4.842512e-04 & 2.87 & 1.699877e-03 & 2.69 & 1.261068e-02 & 2.22  \cr \hline 
0.0884 & 8.728429e-05 & 2.47 & 3.148312e-04 & 2.43 & 2.826181e-03 & 2.16  \cr \hline 
\end{tabular}
\caption{The RMSE of spline solutions using the pair
$X_h=S^{-1}_3(\triangle_\ell), M_h=S^{-1}_1(\triangle_\ell)$ of spline spaces for $
\ell=0, 1, 2, 3, 4$ of PDE (\ref{PDEex3}) based on uniform
triangulations in Figure~\ref{tris}\label{ndexd3}}
\end{table}

\begin{table}[thbp]
\centering
\begin{tabular}{|c|c|r|c|r|c|r|}
\hline $|\triangle|$ & $u-S_u$ & rate & $\nabla(u-S_u)$ & rate &
$\nabla^2(u- S_u)$ & rate \cr\hline
0.7071 & 1.297545e-03 & 0.00 & 3.501439e-03 & 0.00 & 2.743879e-02 & 0.00  \cr \hline 
0.3536 & 7.214648e-05 & 4.17 & 2.100133e-04 & 4.06 & 3.458414e-03 & 2.99  \cr \hline 
0.1768 & 4.215572e-06 & 4.10 & 1.213607e-05 & 4.11 & 3.941213e-04 & 3.13  \cr \hline 
0.0884 & 2.399293e-07 & 4.14 & 7.071917e-07 & 4.10 & 4.553490e-05 & 3.11  \cr \hline 
\end{tabular}
\caption{The RMSE of spline solutions using the pair
$X_h=S^{-1}_4(\triangle_\ell), M_h=S^{-1}_2(\triangle_\ell)$ of spline spaces for $
\ell=0, 1, 2, 3, 4$ of PDE (\ref{PDEex3}) based on uniform
triangulations in Figure~\ref{tris}\label{ndexd4}}
\end{table}

\begin{table}[htpb]
\centering
\begin{tabular}{|c|c|r|c|r|c|r|}
\hline $|\triangle|$ & $u-S_u$ & rate & $\nabla(u-S_u)$ & rate &$\nabla^2(u- S_u)$ & rate \cr\hline
0.7071 & 4.204985e-05 & 0.00 & 2.212172e-04 & 0.00 & 2.973071e-03 & 0.00  \cr \hline 
0.3536 & 2.322110e-06 & 4.18 & 8.539613e-06 & 4.70 & 1.945146e-04 & 3.92  \cr \hline 
0.1768 & 1.418182e-07 & 4.03 & 3.877490e-07 & 4.46 & 1.172927e-05 & 4.04  \cr \hline 
0.0884 & 8.751103e-09 & 4.02 & 2.151534e-08 & 4.17 & 7.033378e-07 & 4.06  \cr \hline 
\end{tabular}
\caption{The RMSE of spline solutions using the pair
$X_h=S^{-1}_5(\triangle_\ell), M_h=S^{-1}_3(\triangle_\ell)$ of spline spaces for $
\ell=0, 1, 2, 3, 4$ of PDE (\ref{PDEex3}) based on uniform
triangulations in Figure~\ref{tris}\label{ndexd5}}
\end{table}

\begin{table}[thbp]
\centering
\begin{tabular}{|c|c|r|c|r|c|r|}
\hline $|\triangle|$ & $u-S_u$ & rate & $\nabla(u-S_u)$ & rate &
$\nabla^2(u- S_u)$ & rate \cr\hline
0.7071 & 3.172567e-06 & 0.00 & 1.554581e-05 & 0.00 & 2.401833e-04 & 0.00  \cr \hline 
0.3536 & 4.869340e-08 & 6.03 & 2.888992e-07 & 5.75 & 8.574254e-06 & 4.79  \cr \hline 
0.1768 & 7.439441e-10 & 6.03 & 4.886180e-09 & 5.89 & 2.752530e-07 & 4.95  \cr \hline 
0.0884 & 7.196367e-12 & 6.69 & 7.790660e-11 & 5.97 & 8.597503e-09 & 5.00  \cr \hline  
\end{tabular}
\caption{The RMSE of spline solutions using the pair
$X_h=S^{-1}_6(\triangle_\ell), M_h=S^{-1}_4(\triangle_\ell)$ of spline spaces for $
\ell=0, 1, 2, 3, 4$ of PDE (\ref{PDEex3}) based on uniform
triangulations in Figure~\ref{tris}\label{ndexd6}}
\end{table}

\begin{table}[thbp]
\centering
\begin{tabular}{|c|c|r|c|r|c|r|}
\hline $|\triangle|$ & $u-S_u$ & rate & $\nabla(u-S_u)$ & rate &
$\nabla^2(u- S_u)$ & rate \cr\hline
0.7071 & 2.593663e-07 & 0.00 & 7.442269e-07 & 0.00 & 1.217246e-05 & 0.00  \cr \hline 
0.3536 & 4.280178e-09 & 5.92 & 8.943363e-09 & 6.38 & 2.025671e-07 & 5.90  \cr \hline 
0.1768 & 7.233371e-11 & 5.89 & 1.274428e-10 & 6.13 & 3.166771e-09 & 5.99  \cr \hline  
\end{tabular}
\caption{The RMSE of spline solutions using the pair
$X_h=S^{-1}_7(\triangle_\ell), M_h=S^{-1}_5(\triangle_\ell)$ of spline 
spaces for $\ell=1, 2, 3$ of PDE (\ref{PDEex3}) based on uniform
triangulations in Figure~\ref{tris}\label{ndexd7}}
\end{table}

\begin{table}[thbp]
\centering
\begin{tabular}{|c|c|r|c|r|c|r|}
\hline $|\triangle|$ & $u-S_u$ & rate & $\nabla(u-S_u)$ & rate &
$\nabla^2(u- S_u)$ & rate \cr\hline
0.7071 & 8.731770e-09 & 0.00 & 3.136707e-08 & 0.00 & 6.080104e-07 & 0.00  \cr \hline 
0.3536 & 5.745609e-11 & 7.25 & 1.807841e-10 & 7.44 & 5.399368e-09 & 6.81  \cr \hline 
0.1768 & 1.079912e-11 & 2.41 & 2.701168e-11 & 2.74 & 8.501408e-10 & 2.69  \cr \hline  
\end{tabular}
\caption{The RMSE of spline solutions using the pair
$X_h=S^{-1}_8(\triangle_\ell), M_h=S^{-1}_6(\triangle_\ell)$ of spline 
spaces for $\ell=1, 2, 3$ of PDE (\ref{PDEex3}) based on uniform
triangulations in Figure~\ref{tris}\label{ndexd8}}
\end{table}

In Table~\ref{ndexd8}, the expected accurate solutions for degree 
$8$ at the triangulation with size $0.1768$ were not able to achieve
due to the limitation of our iterative algorithm CImin.m. 
Numerical convergence in Tables~\ref{ndexd2}
--\ref{ndexd8} provide more evidence that the convergence 
$|u-u_h|_{H^2(\Omega)}$ is of order $d-1$ for $d=2, 3, \cdots, 8$. 
In addition, the convergence rate for $\|u-u_h\|_{L^2(\Omega)}$ can be
better than $d-1$ for various $d$.  
\end{example}

\subsection{Numerical Results of PDE in (\ref{model}) with nonzero $c$}
In this subsection, we present some numerical results from our bivariate 
spline method for numerical solution of the PDE in (\ref{model}) with
nonzero $c$. We use three examples to demonstrate that our method is
effective and efficient no matter the PDE coefficients are smooth 
or not smooth and the solutions are smooth or not so smooth. 

\begin{example}
\label{ex51} We begin with a 2nd order elliptic equation with
smooth coefficients and smooth solution $u =\sin(\pi x)\sin(\pi y)$ 
which satisfies the following partial differential equation:
\begin{equation}
\label{PDEex51} 3\frac{\partial^2}{\partial x^2} u +
2\frac{\partial^2}{\partial x\partial y} u
+2\frac{\partial^2}{\partial y^2} u - 
(1+x^2+y^2)u = f(x,y), \quad (x,y)\in
\Omega\subset \mathbb{R}^2,
\end{equation}
where $\Omega$ is a standard  square domain  $[-1,1]^2$ which is split
into 4 equal sub-squares and each sub-square is split into 2 triangles
to form an initial triangulation $\triangle_0$. Let $\triangle_\ell$ 
be the $\ell$th uniform refinement of $\triangle_0$. 
 
\begin{table}[thbp]
\centering
\begin{tabular}{|c|c|r|c|r|c|r|}
\hline $|\triangle|$ & $u-S_u$ & rate & $\nabla(u-S_u)$ & rate &
$\nabla^2(u- S_u)$ & rate \cr\hline 
0.7071 & 7.148997e-04 & 0.00 & 5.688698e-03 & 0.00 & 7.513832e-02 & 0.00  \cr \hline 
0.3536 & 2.651667e-05 & 4.75 & 1.861396e-04 & 4.93 & 4.596716e-03 & 3.99  \cr \hline 
0.1768 & 1.257317e-06 & 4.40 & 6.093065e-06 & 4.93 & 2.814578e-04 & 4.03  \cr \hline 
0.0884 & 7.088746e-08 & 4.15 & 2.550376e-07 & 4.58 & 1.753997e-05 & 4.00  \cr \hline 
\end{tabular}
\caption{The RMSE of spline solutions using the pair 
$X_h=S^{-1}_5(\triangle_\ell), M_h=S^{-1}_5(\triangle_\ell)$ of spline 
spaces for $\ell=1, 2, 3, 4$ of PDE (\ref{PDEex51})\label{Nex51d5}}
\end{table} 
\end{example}

\begin{example}
\label{ex52} In this example, we use our spline method to solve the
following PDE with non-differentiable coefficients, but smooth solution.
\begin{equation}
\label{PDEex52} a(x,y)\frac{\partial^2}{\partial x^2} u +
b(x,y)\frac{\partial^2}{\partial x\partial y} u
+c(x,y)\frac{\partial^2}{\partial y^2} u -(1+x^2+y^2)u 
= f(x,y), \quad (x,y)\in
\Omega\subset \mathbb{R}^2
\end{equation}
where $a(x,y)=1+|x|, b(x,y)=(xy)^{1/3}, c(x,y)=1+|y|$ and $\Omega$
is a standard  domain  $[-1, 1]^2$. We use 
$u=\sin(\pi x)\sin(\pi y)$ as the exact solution. The same 
triangulations $\triangle_\ell$ as in Example~\ref{ex51} will be 
used.

\begin{table}[thbp]
\centering
\begin{tabular}{|c|c|r|c|r|c|r|}
\hline $|\triangle|$ & $u-S_u$ & rate & $\nabla(u-S_u)$ & rate &
$\nabla^2(u- S_u)$ & rate \cr\hline 
0.7071 & 5.906274e-04 & 0.00 & 5.894580e-03 & 0.00 & 9.080845e-02 & 0.00  \cr \hline 
0.3536 & 1.155544e-05 & 5.68 & 1.785330e-04 & 5.05 & 5.673534e-03 & 3.97  \cr \hline 
0.1768 & 3.265019e-07 & 5.15 & 4.834318e-06 & 5.21 & 3.150799e-04 & 4.15  \cr \hline 
0.0884 & 1.568269e-08 & 4.38 & 1.463029e-07 & 5.05 & 1.866297e-05 & 4.07  \cr \hline 
\end{tabular}
\caption{The RMSE of spline solutions using the pair 
$X_h=S^{-1}_5(\triangle_\ell), M_h=S^{-1}_5(\triangle_\ell)$ of spline 
spaces for $\ell=1, 2, 3, 4$ of PDE (\ref{PDEex52})\label{Nex52d5}}
\end{table}
\end{example}

\begin{example}
\label{ex53} In this example, we show the performance of our spline
solutions for a PDE with  nondifferentiable coefficients and
nonsmooth exact solution 
 $u=xy(e^{1-|x|}-1)(e^{1-|y|}-1)$ which satisfies
\begin{equation}
\label{PDEex53} 2\frac{\partial^2}{\partial x^2} u +
2\sign(x)\sign(y)\frac{\partial^2}{\partial x\partial y} u
+2\frac{\partial^2}{\partial y^2} u-(1+x^2+y^2)u = f(x,y), \quad (x,y)\in
\Omega\subset \mathbb{R}^2
\end{equation}
where $u=0$ on the boundary of  $\Omega=[-1, 1]\times [-1, 1]$ as in
\cite{SS13}. 
Note that the solution is in $H^2(\Omega)$, but not
continuously twice differentiable.  The same 
triangulations $\triangle_\ell$ as in Example~\ref{ex51} will be 
used and $S^1_5(\triangle_\ell)$ will be used 
to solve the PDE in (\ref{PDEex53}). The RMSE for spline approximation
to the exact solution is shown in Table~\ref{Nex53d5}.

\begin{table}[thbp]
\centering
\begin{tabular}{|c|c|r|c|r|c|r|}
\hline $|\triangle|$ & $u-S_u$ & rate & $\nabla(u-S_u)$ & rate &
$\nabla^2(u- S_u)$ & rate \cr\hline 
0.7071 & 2.914706e-02 & 0.00 & 2.813852e-01 & 0.00 & 4.122223e+00 & 0.00  \cr \hline 
0.3536 & 8.047145e-04 & 5.18 & 1.287751e-02 & 4.46 & 3.279903e-01 & 3.63  \cr \hline 
0.1768 & 2.963898e-05 & 4.76 & 3.942771e-04 & 5.02 & 1.893540e-02 & 4.06  \cr \hline 
0.0884 & 1.403774e-06 & 4.40 & 1.307268e-05 & 4.91 & 1.139668e-03 & 4.05  \cr \hline 
\end{tabular}
\caption{The RMSE of spline solutions using the pair 
$X_h=S^{-1}_5(\triangle_\ell), 
M_h=S^{-1}_5(\triangle_\ell)$ of spline spaces for $
\ell=1, 2, 3, 4$ of PDE (\ref{PDEex53})\label{Nex53d5}}
\end{table}

\begin{table}[thbp]
\centering
\begin{tabular}{|c|c|r|c|r|c|r|}
\hline $|\triangle|$ & $u-S_u$ & rate & $\nabla(u-S_u)$ & rate &
$\nabla^2(u- S_u)$ & rate \cr\hline 
0.7071 & 3.494139e-06 & 0.00 & 1.538901e-05 & 0.00 & 2.054258e-04 & 
0.00  \cr \hline 
0.3536 & 7.686885e-08 & 5.51 & 3.531160e-07 & 5.45 & 7.914461e-06 & 
4.70  \cr \hline 
0.1768 & 1.370000e-09 & 5.81 & 6.565436e-09 & 5.75 & 2.731226e-07 & 
4.86  \cr \hline 
0.0884 & 1.598556e-11 & 6.42 & 7.840289e-11 & 6.39 & 8.199009e-09 & 
5.06  \cr \hline 
\end{tabular}
\caption{The RMSE of spline solutions using the pair 
$X_h=S^{-1}_6(\triangle_\ell), 
M_h=S^{-1}_6(\triangle_\ell)$ of spline spaces for $
\ell=1, 2, 3, 4$ of PDE (\ref{PDEex53})\label{Nex53d6}}
\end{table}

\begin{table}[thbp]
\centering
\begin{tabular}{|c|c|r|c|r|c|r|}
\hline $|\triangle|$ & $u-S_u$ & rate & $\nabla(u-S_u)$ & rate &
$\nabla^2(u- S_u)$ & rate \cr\hline 
0.7071 & 6.099647e-08 & 0.00 & 4.914744e-07 & 0.00 & 1.004845e-05 & 
0.00  \cr \hline 
0.3536 & 9.745659e-10 & 5.97 & 4.297639e-09 & 6.84 & 1.659379e-07 & 
5.92  \cr \hline 
0.1768 & 9.091448e-12 & 6.74 & 5.183151e-11 & 6.37 & 3.038997e-09 & 
5.78  \cr \hline 
\end{tabular}
\caption{The RMSE of spline solutions using the pair 
$X_h=S^{-1}_7(\triangle_\ell), 
M_h=S^{-1}_7(\triangle_\ell)$ of spline spaces for $
\ell=1, 2, 3, 4$ of PDE (\ref{PDEex53})\label{Nex53d7}}
\end{table}

In addition, we have also experimented the convergence of our bivariate 
spline method over nonconvex domains. The bivariate spline method 
works very well. Due to the space limit, we omit these 
numerical results. 
\end{example}


\end{document}